\documentclass{cmslatex}
\usepackage[paperwidth=7in, paperheight=10in, margin=.875in]{geometry}
 \usepackage[backref,colorlinks,linkcolor=red,anchorcolor=green,citecolor=blue]{hyperref}
\usepackage{amsfonts,amssymb}
\usepackage{amsmath}
\usepackage{graphicx}
\usepackage{cite}
\usepackage{enumerate}
\usepackage{etex}
\usepackage[latin1]{inputenc}
\usepackage[T1]{fontenc}
\usepackage{amsfonts}
\usepackage{amssymb, amsmath}
\usepackage{geometry}
\usepackage{paralist}
\usepackage{lscape}
\usepackage[english]{babel} 
\usepackage{ifthen}
\usepackage{bbm}		
\usepackage{graphicx}
\usepackage{lmodern}
\usepackage{algorithmic}
\usepackage{float}
\usepackage{wrapfig} 
\usepackage{textcomp}
\usepackage{graphicx}
\usepackage{cancel}
\usepackage{mathrsfs}
\usepackage{upgreek}
\usepackage{url}
\usepackage{tikz}	
\usepackage{textpos}
\usepackage{mathtools} 
\usepackage{ragged2e}
\usepackage{tabularx}
\usepackage{nicefrac}
\usepackage{soul}
\usepackage{yhmath}

\renewcommand{\theequation}{\arabic{section}.\arabic{equation}}

\def \h{\textup{\textbf{h}}}
\def \u{\textup{\textbf{u}}}
\def \T{\textup{T}}
\def \d{\textup{T},}
\def \entropy{\eta}
\def \entropyflux{\psi}
\def \R{\mathcal{R}}
\def \Veq{V_{\textup{eq}}}
\def \HAR{\textup{\textbf{H}}}
\def \d{\,\textup{d}}
\def \WAR{\textup{\textbf{W}}}
\def \vCorli{\mathcal{V}_C}
\def \Vtilde{\widetilde{\mathcal{V}}}
\def \omegaTilde{\widetilde{\omega}}
\def \hA{\textup{\textbf{h}}_A}
\def \hD{\textup{\textbf{h}}_D}
\def \HCorli{H_C}
\def \hCorli{\textup{\textbf{h}}}
\def \muCorli{\mu_C}
\def \v1{v^{(1)}}

   \allowdisplaybreaks
\begin{document}
 \title{Relaxation and stability analysis of a  third-order multiclass traffic flow model}
 	

          \author{Stephan Gerster\thanks{Department of Mathematics ``Guido Castelnuovo'', Sapienza~University~of~Rome, (stephan.gerster@uniroma1.it).}
          \and Giuseppe Visconti\thanks{Department of Mathematics ``Guido Castelnuovo'', Sapienza~University~of~Rome, (giuseppe.visconti@uniroma1.it).}}

         \pagestyle{myheadings} \markboth{Third-order taffic flow model}{Stephan Gerster and Giuseppe Visconti} \maketitle

          \begin{abstract}
Traffic flow modeling spans a wide range of mathematical approaches, from microscopic descriptions of individual vehicle dynamics to macroscopic models based on aggregate quantities. A fundamental challenge in macroscopic modeling lies in the closure relations, particularly in the specification of a traffic hesitation function in second-order models like Aw-Rascle-Zhang.
In this work, we propose a  third-order hyperbolic traffic model in which the hesitation evolves as a driver-dependent dynamic quantity. 
Starting from a microscopic formulation, we relax the standard assumption by introducing an evolution law for the hesitation. This extension allows to incorporate hysteresis effects, modeling the fact that drivers respond differently when accelerating or decelerating, even under identical local traffic conditions. 
Furthermore, various relaxation terms are introduced. These allow us to establish relations to the Aw-Rascle-Zhang model and other traffic flow models.
          \end{abstract}
          


\section{Introduction} \label{sec:intro}
The modeling of traffic flow has inspired a rich variety of mathematical approaches, each capturing different levels of detail and operating at different scales of representation. In general, these models can be classified as microscopic~\cite{Gong:2023,Piu:2022,Hayat:2023,Albeaik:2022}, mesoscopic~\cite{Marques:2013,Iannini:2016,Lu:2025}, or macroscopic~\cite{Piu:2023,Herty:2018,Song:2019}. Microscopic models describe the dynamics of individual vehicles, accounting for detailed interactions among drivers. Macroscopic models, on the other hand, treat traffic flow analogously to fluid dynamics, focusing on aggregate quantities such as vehicle density and mean velocity. Mesoscopic models lie between these two scales, describing the behavior of groups of vehicles while retaining some statistical characteristics of individual motion.

A key aspect in the development and understanding of traffic models is the ability to link different scales through suitable scaling limits, allowing for a consistent transition from microscopic dynamics to macroscopic behavior~\cite{Borsche:2018,Burger:2018,Cristiani:2016,Cardaliaguet:2021}. 
For instance, 
small variability of microscopic velocities at equilibrium leads to 
stop-and-go waves, which are macroscopically observed in traffic flow~\cite{Herty:2020,Goatin:2024}. In this work, we focus on the microscopic and macroscopic scales, which represent the finest and coarsest levels of description, respectively.

Microscopic models, while highly detailed, are less predictive on large scales and are often computationally expensive, making them impractical for real-time applications such as traffic monitoring and control. In contrast, macroscopic models offer a more tractable and efficient framework for such purposes but suffer from limitations related to closure relations, which must be specified to complete the system of equations.

First-order macroscopic models, most notably the Lighthill-Whitham-Richards (LWR) model~\cite{Lighthill:1955,Richards:1956}, assume homogeneous behavior of various drivers and do not account for different types of vehicles, such as a fast car or a slow truck. 
Furthermore, they rely on 
a prescribed relationship between vehicle density and velocity. These assumptions imply that the traffic flow is always in equilibrium, with instantaneous adjustment of speed to local conditions. 
These simplifications often fail to capture realistic driving behavior. 
Two approaches have emerged to address this problem. 

The first solution concept is to replace
the equilibrium velocity of a first-order model with a dynamic variable that is described by an additional equation. 
The second solution concept incorporates the behaviors of different types of vehicles and drivers using multiclass models. 
Both approaches have in common that they result in higher-order systems that account for the different classes of drivers and additional unknowns that describe changes in the dynamic velocity. 

One of the most widely noted contributions to the first research direction, namely, describing the dynamics of velocity, was provided by Aw and Rascle~\cite{Aw:2000}, and independently by Zhang~\cite{Zhang:2002}, who proposed a second-order macroscopic (ARZ) model. 
Inspired by gas dynamics, this model introduces an additional equation that accounts for vehicle acceleration and inertia. This model addresses the shortcomings of earlier second-order approaches~\cite{Daganzo:1995}, but still requires a closure relation in the form of traffic hesitation, which is typically modeled as a monotonically increasing function of the density. Various formulations for traffic hesitation have been proposed in the literature, and these macroscopic models may be derived either from physical principles~\cite{Berthelin:2008} or via asymptotic limits of microscopic models~\cite{Aw:2002}.

Indeed, while the LWR model can be obtained as the hydrodynamic limit of first-order microscopic models~\cite{Holden:2018},  the ARZ model arises from the limit of second-order microscopic models~\cite{Aw:2002}, such as the combined Follow-the-Leader and Bando models~\cite{Gazis:1961,Bando:1995}.
In the microscopic setting, the closure is embedded in the vehicle dynamics, First-order models prescribe a velocity that depends on the headway (the distance to the vehicle ahead), while second-order models incorporate interaction forces or accelerations that are typically proportional to a function of the local density. 
The central difficulty lies in the need for a closure relation that must be specified \textit{a priori}, and its choice significantly influences the qualitative behavior of the model. 

Approaches that address the second solution concept, namely the inclusion of behaviors of different types of vehicles and drivers, typically include one or more additional equations to describe features of additional classes~\cite{Benzoni2003,Bagnerini:2003}. Since these are characteristics of each vehicle, they do not depend on time in Lagrangian coordinates. 

In this work, we propose a novel third-order hyperbolic traffic model (TOM), which is inspired by both solution concepts. More precisely, hesitation is no longer prescribed as a fixed function of the density; instead, the hesitation evolves dynamically, introducing an additional equation. Our approach is based on the derivation of the ARZ model from microscopic Follow-the-Leader and Bando-type models, where a key step involves the introduction of a new variable that is a conserved quantity in Lagrangian coordinates, namely~$\omega=v+P(\rho)$, where $P(\rho)$ represents the traffic hesitation.  
Our approach seeks to make this choice more suited for real-world problems. Specifically, we relax this assumption by making the hesitation $P$ no longer constrained to be a static function of the density, and instead treat it as a dynamic variable with its own evolution in time.

Specifically, we allow hesitation to depend on hysteresis effects in traffic, reflecting the idea that drivers' responses to changes in headway or velocity are not instantaneous and may be influenced by past states. 
Hence, we introduce a driver- or car-dependent property similarly to the second solution concept. 

This allows us to derive a third-order microscopic model from which we obtain a corresponding macroscopic model in Lagrangian coordinates, and subsequently transform it into the Eulerian framework via a suitable change of variables.

The hysteresis model we consider is inspired by the formulation introduced by Corli and Fan~\cite{Corli:2019,Corli:2023}, which captures the memory-dependent nature of driver behavior, particularly in congested or stop-and-go traffic regimes. We argue that modeling hesitation through hysteresis provides a more accurate representation of real traffic dynamics. In particular, it accounts for the asymmetry in driver reactions.  Empirical observations suggest that drivers tend to accelerate more cautiously than they decelerate, even under identical traffic conditions -- a behavior that cannot be reproduced by hesitation laws depending solely on the current density. 
Then, we will introduce relaxation terms in the third-order model 
as in the inhomogeneous ARZ model, where  a relaxation term allows drivers to achieve the equilibrium speed~\cite{Greemberg2002}. 
In the small relaxation limit, the ARZ model approaches the LWR model, which can be obtained by means of a Chapman-Enskog-type expansion. 
Here, the stability and wellposedness of solutions to the inhomogeneous ARZ model are governed by the study of the sign of the diffusion coefficient. 
More precisely, the sub-characteristic condition~\cite{Jin1995}  holds if and only if the advection-diffusion equation is dissipative, which is obtained via a first-order Chapman-Enskog-type expansion, provided that the system is of size two~\cite[Th.~3.1]{Chen1994}. This statement, however, does~\emph{not} necessarily hold for hyperbolic relaxation problems of general size.

Indeed, relaxation and stability results are well-established for second-order systems, whereas results for higher-order models are still partial.
In fact, the local equilibrium approximation of a system of size three may not even be hyperbolic~\cite[Sec.~2]{Chen1994}. 
Whenever a loss of hyperbolicity does not arise in the local equilibrium approximation, it makes sense to seek a first-order correction. 
However, it is generally an open question whether the first-order correction is dissipative even if the characteristic speeds of the local equilibrium  are not excessive~\cite{Chen1994}. 

We contribute to this research gap by conducting a stability analysis for various relaxation terms. 
In particular, we examine the relaxation to the ARZ model in terms of a conservative and non-conservative form. 
We obtain, for the conservative form of our proposed model of size three, stabilization results that are similar to those of~\cite[Th.~3.1]{Chen1994}, which hold only for second-order systems. 
Namely, the first-order correction leads to a dissipative advection-diffusion equation if and only if the characteristic curves of the local equilibrium do not exceed those of the relaxed system. 
On the other hand, a relaxation with respect to a non-conservative form results in a system that is of Hamilton-Jacobi-type. Although a diffusion coefficient cannot be investigated in this case, we still obtain results in terms of the characteristic speeds. 
A similar problem arises when the system is relaxed to a family of velocity curves. Then, a Chapman-Enskog-type expansion leads to a non-symmetric diffusion matrix, where the equivalence~\cite[Th.~3.1]{Chen1994} also does not hold. In this case, we obtain a relation between the relaxed system and the zero-relaxation limit in the sense of vanishing second-order terms.\\

\noindent
This paper is structured as follows. 
We begin in Section~\ref{SectionMicro} with the motivation and derivation of the newly proposed traffic model at the microscopic level in terms of ordinary differential equations. 
In Section~\ref{SectionMacroscopicPiu}
the hydrodynamic limit is stated in terms of a macroscopic hyperbolic system in Lagrangian coordinates. The system is then transformed in Eulerian coordinates. Various representations of the resulting macroscopic model,  the wave structure and entropy-entropy flux pairs are stated. In particular, a Temple class system is obtained. Section~\ref{SectionRelaxation} investigates the relaxation towards the ARZ model and towards multiclass traffic flow models with parameterized families of velocity curves. Finally, 
the Riemann problem is discussed in~Section~\ref{SectionNumerics}.

\newpage
\section{Microscopic description of traffic flow}\label{SectionMicro}
Most microscopic traffic flow models are based on a \textit{car-following} ansatz, where the evolution of the position~$x_i(t)$ and the velocity~$v_i(t)$ of a car~$i$ at time~${t\geq 0}$ is described by ordinary differential equations~(ODEs) and depends only on the $i+1$ vehicle ahead. 

\subsection{Follow-the-Leader models}\label{SecARZmicro}
A microscopic description of the car dynamics~\cite{Aw:2002} can be given by the following system of ODEs:
\begin{equation}\label{FtL}
	\begin{cases}
		\begin{aligned}
			\dot{x}_i&=v_i, \\
			\dot{v}_i&=c_\gamma \frac{\dot{x}_{i+1}-\dot{x}_{i}}{(x_{i+1}-x_i)^{\gamma+1}},
		\end{aligned}
	\end{cases}
\end{equation}
where~$\dot{x}(t) = x'(t)$ denotes the time derivative. 
The local density around vehicle~$i$, which is a dimensionless quantity, and its inverse, the local normalized specific volume, are defined by
$$
\rho_i\coloneqq
\frac{\Delta X}{x_{i+1}-x_i}
\quad\text{and}\quad
\tau_i\coloneqq \frac{1}{\rho_i},
$$
where~$\Delta X$ is the length of a car.
The choices~$c_\gamma\coloneqq\gamma (\Delta  X)^\gamma$ and $P(\rho)\coloneqq\rho^\gamma$ lead to the relation
$$
-\dot{v}_i
=
-
c_\gamma \frac{\dot{x}_{i+1}-\dot{x}_{i}}{(x_{i+1}-x_i)^{\gamma+1}}
=
c_\gamma
\frac{1}{(x_{i+1}-x_i)^{\gamma-1}}
\bigg(
-
\frac{\dot{x}_{i+1}-\dot{x}_i}{(x_{i+1}-x_i)^2}
\bigg)
=
P'(\rho_i)\dot{\rho}_i
=
\dot{P(\rho_i)},
$$
where $\dot{P(\rho_i)} \coloneqq
\frac{\textup{d}}{\textup{d}t}
P\big( \rho_i(t) \big)$
 denotes the composite time derivative. Hence, the relation~$\dot{\omega}_i=0$ holds for~$\omega_i\coloneqq v_i+P(\rho_i)$. 
Since the quantity~$\omega_i(t)=\omega_i(0)$ is constant over time, it is a driver-dependent quantity that corresponds to car~$i$. 

\subsection{Multiclass models}\label{SecMicroMulticlass}
The idea of multiclass models~\cite{Bagnerini:2003} is to associate to each driver  an additional~information~$h_i$,  which accounts for different driving characteristics. For instance, a car accelerates faster than a truck. 
Then, the Follow-the-Leader system~\eqref{FtL} is augmented with an additional equation and reads as
\begin{equation}\label{MulticlassMicroscopic}
	\begin{cases}
		\begin{aligned}
			\dot{x}_i&=v_i, \\
			\dot{\omega}_i&=0, \\
			\dot{h}_i &= 0
		\end{aligned}
	\end{cases}
	\text{with}
	\quad \
	\omega_i\coloneqq
	v_i+q(\rho_i,h_i).
\end{equation}
Since the driver-dependent information~$h_i(t)=h_i(0)$ is constant over time, it can be seen as a known quantity that acts similarly to a parameter, which parametrizes the term~$q(\cdot,h_i)$.

\subsection{The third-order model (TOM)}\label{SecMicroPiu}
We propose a  model that is quite related to a multiclass model. It is based on the ansatz
\begin{equation}\label{PiuMicroscopic}
	\begin{cases}
		\begin{aligned}
			\dot{x}_i&=v_i, \\
			\dot{\omega}_i&=0, \\
			\dot{h_i} &= \dot{\h(\rho_i)}
		\end{aligned}
	\end{cases}
	\text{with}
	\qquad
	\omega_i\coloneqq
	v_i+q(h_i).
\end{equation}
It is motivated by the following point of criticism regarding the choices~$\omega_i=
v_i+P(\rho_i)$ and~$\omega_i=
v_i+q(\rho_i,h_i)$ in Subsections~\ref{SecARZmicro} and~\ref{SecMicroMulticlass}. Namely, 
closure relations $\rho_i \mapsto P(\rho_i)$ and $\rho_i \mapsto q(\rho_i,\cdot)$ are required. For instance, the relation~$P(\rho)=\rho^\gamma$ in Subsection~\ref{SecARZmicro}, which is needed to obtain the expression~$\dot{\omega}_i=0$, is relatively restrictive, since the Follow-the-Leader term could also have a more general formulation~\cite{Herty:2020}. 
By rewriting the second equation of the proposed model~\eqref{PiuMicroscopic} in terms of velocity, i.e.
$$
\dot{v}_i= -
\dot{q(h_i)}=
-q'(h_i)
\dot{h_i}
=
-q'(h_i)
\dot{\h(\rho_i)}
=
-q'(h_i)
\h'(\rho_i)
\dot{\rho_i}
=-
q'(h_i)
\h'(\rho_i)
\rho_i
\frac{v_{i+1}-v_i}{x_{i+1}-x_i},
$$
we observe that there are no restrictive assumptions on the functions~$q$ and $\h$ required, which must be specified for the  proposed ansatz in~\eqref{PiuMicroscopic}, with the exception of the non-restrictive assumption~${q'(h_i)\h'(\rho_i)>0}$. 
Namely, the term~$q'(h_i)\h'(\rho_i)$ generalizes the derivative~$P'(\rho_i)>0$ in the follow-the-leader models in Section~\ref{SecARZmicro}. 
Therefore, the requirement of a closure does not occur for the choice~$\omega_i=
v_i+q(h_i)$, which we propose. However, the resulting ODE~$\big\{ \dot{x}_i=v_i,\,\dot{\omega}_i =0 \big\}$ must be closed by an additional equation, namely~$ \dot{h}_i = \dot{\h}(\rho_i)$. For this reason, the model~\eqref{PiuMicroscopic} arises rather from a modeling step of the dynamics than from a closure mapping or parameterization as in previous models. In particular, a reasonable choice for~$\h(\rho)$ is given in~\cite{Corli:2019,Corli:2023}, where~$h$ accounts for the hesitation of drivers.

\section{The macroscopic TOM model}
\label{SectionMacroscopicPiu}
The macroscopic description of the microscopic Follow-the-Leader and multiclass models in Section~\ref{SecARZmicro} and~\ref{SecMicroMulticlass} has been ensured by a hydrodynamic limit~\cite{Aw:2002,Bagnerini:2003}. The same arguments yield the corresponding formulation of the proposed microscopic model \eqref{PiuMicroscopic} on a macroscopic level, when the number of cars tend to infinity and their size~$\Delta X\rightarrow 0$ shrinks to zero. It reads in Lagrangian coordinates as
\begin{equation}\label{PiuLagrange}
	\begin{cases}
		\begin{aligned}
			\partial_T \tau(T,X) - \partial_X v(T,X) \ \ &=0,\\
			\partial_T \Big(v(T,X)+ q \big(h(T,X)\big)\Big) &=0,\\
			\partial_T \Big(h(T,X)-\h\big(\rho(T,X)\big)\Big) &=0
		\end{aligned}
	\end{cases}
	\text{for}\quad
	\tau\coloneqq
	\frac{1}{\rho}.
\end{equation}	
Here, the variables $(T,X)$ denote the Lagrangian time and space coordinates. Those are related to the Eulerian $(t,x)$ by~$X(t,x) = \int^x \rho(t,\xi) \d \xi$, which implies 
the transforms
\begin{equation}\label{TransformsLagrange}
	\partial_t X(t,x)=-\rho v,\ \ 
	\partial_x X(t,x)=\rho,\ \ 
	\partial_T x(T,X)=v,\ \
	\partial_X x(T,X)=\frac{1}{\rho}
	\ \ \text{and}\ \ 
	T=t.
\end{equation}
Note that the variable~$X$ is not a mass, but
describes the total length occupied by cars up to the point in space~$x$, since the local density $\rho$ is a dimensionless quantity. 
Likewise, the driver-dependent quantity~$h$ is dimensionless, while the function~$q$ has the physical dimension of a velocity.\\

\begin{remark}
   	Using initial values~$
	h-\h(\rho)=h_0-\h(\rho_0)
	$ 
	and~$
	v+q(h)=v_0+q(h_0)
	$, system~\eqref{PiuLagrange} reduces in Lagrangian coordinates to the scalar equation
	\begin{align*}\label{PiuScalar}
		&	\partial_T \tau(t,X)-\partial_X v\big(\tau(t,X);X\big)
		=0 \\
		&\begin{aligned}
			&	\text{with}
			&	v\big(\tau(t,X);X\big)
			&\coloneqq
			v_0(X)+q\big(h_0(X)\big)-q\Big(
			h_0(X)-\h\big(\rho_0(X)\big)+\h\big(\tau(t,X)^{-1}\big)
			\Big) \\
			&
			\text{and}
			&	v'\big(\tau(t,X);X\big)
			&\; =q'\Big(
			h_0(X)-\h\big(\rho_0(X)\big)+\h\big(\tau(t,X)^{-1}\big)\Big)
			\h'\big(\tau(t,X)\big)
			\tau(t,X)^{-2}.
		\end{aligned}
	\end{align*}
\end{remark}

\medskip
\noindent
Applying\footnote{The calculations follow the lines of~\cite{Aw:2002,Bagnerini:2003} and are included in the appendix for the sake of the reader.} the transforms~\eqref{TransformsLagrange},   
the system~\eqref{PiuLagrange} in Eulerian coordinates reads as
\begin{equation}\label{PiuS1}\tag{$\mathcal{S}1$}
	\begin{cases}
		\begin{aligned}
			\partial_t \rho + \partial_x(\rho v)&=0,\\
			\partial_t v+ \Big(v-q'(h)\h'(\rho)\rho\Big)\partial_x v
			&=0, \\
			\partial_t h + v \partial_xh&= 
			\partial_t \h(\rho) + v \partial_x \h(\rho). 
		\end{aligned}
	\end{cases} 
\end{equation}
Due to the first equation, we have~$
\partial_t \h(\rho) 
+
v \partial_x \h(\rho)
= -\h'(\rho)\rho \partial_x v
$. Hence, the system~\eqref{PiuS1} can be written in the quasilinear form 
$$
\partial_t \u + Q(\u) \partial_x \u=0 
\quad\text{with}\quad 
Q(\u)
=
\begin{pmatrix}
	v & \rho & 0 \\
	0 & v-q'(h)\h'(\rho)\rho & 0     \\
	0 & \h'(\rho) \rho & v
\end{pmatrix}
\quad\text{and}\quad 
\u=
\begin{pmatrix}
    \rho \\ v \\ h
\end{pmatrix}.
$$
The eigenvalue decomposition
$L(\u) Q(\u) = R(\u) D(\u) 
$ 
reads as
\begin{align*}
	R(\u)&=
	\begin{pmatrix}
		1 & 0 & \rho \\
		0 & 0 & -q'(h)\h'(\rho)\rho   \\
		0 & 1 & \h'(\rho) \rho
	\end{pmatrix}, \\
	L(\u)&=
	\big[q'(h)\h'(\rho)\rho\big]^{-1}
	\begin{pmatrix}
		q'(h)\h'(\rho)\rho & \rho & 0 \\
		0 & \h'(\rho) \rho & q'(h)\h'(\rho)\rho \\
		0 & -1 & 0
	\end{pmatrix}
\end{align*}
with  eigenvalues 
$D(\u) = \diag\big\{ v,v,v-q'(h)\h'(\rho)\rho \big\}$ 
that are labeled as 
$$\lambda_1(\rho,v,h)=v-q'(h)\h'(\rho)\rho
\quad\text{and}\quad \lambda_2(\rho,v,h) = \lambda_3(\rho,v,h) = \lambda_v(v) \coloneqq v.$$
We show in Lemma~\ref{LemmaEquivalencePiu} that the system~\eqref{PiuS1} is equivalent to the system
\begin{equation}\label{PiuS2}\tag{$\mathcal{S}2$}
	\begin{cases}
		\begin{aligned}
			\partial_t \rho + \partial_x(\rho v)&=0,\\
			\partial_t \Big(v+q(h) \Big)
			+v
			\partial_x \Big(v+q(h) \Big)
			&=0, \\
			\partial_t \Big(h-\h(\rho) \Big) + v \partial_x
			\Big(h-\h(\rho) \Big)&= 
			0	\end{aligned}
	\end{cases}  
\end{equation}
and by defining the quantities
$$
z=\rho\big(
v+q(h)
\big)
\quad
\text{and}
\quad
w=\rho\big(h-\h(\rho)\big)
\quad
\text{with}
\quad
v(\rho,z,w)
=
\frac{z}{\rho}-q(h)
\quad
\text{for}
\quad
h=\frac{w}{\rho}
+
\h(\rho)
$$
it has also the conservative form
\begin{equation}\label{PiuS3}\tag{$\mathcal{S}3$}
	\begin{cases}
		\begin{aligned}
			\partial_t \rho + \partial_x\Big(\rho v(\rho,z,w)\Big)&=0,\\
			\partial_tz
			+
			\partial_x \Big(
			\frac{z^2}{\rho}
			-zq(h)
			\Big)
			&=0, \\
			\partial_t w +
			\partial_x \Big(wv(\rho,z,w) 
			\Big)&= 
			0.	\end{aligned}
	\end{cases}  
\end{equation}

\begin{lemma}\label{LemmaEquivalencePiu}
	The formulations~\eqref{PiuS1}~--~\eqref{PiuS3} are equivalent for smooth solutions.	
\end{lemma}

\medskip
\begin{proof}
	The third equation of  system~\eqref{PiuS1} is equivalent to
    $$
	\partial_t 
	\big(h-\h(\rho) \big)
	+
	v
	\partial_x
	\big(h-\h(\rho) \big)=0.$$ 
	Then, the first conservation law implies
	\begin{equation*}
		\begin{aligned}
			0&=
			\Big(
			\partial_t \rho + \partial_x(\rho v)
			\Big)
			\Big( h-\h(\rho)  \Big)
			+
			\rho 
			\partial_t
			\Big( h-\h(\rho) \Big) 
			+
			\rho v
			\partial_x
			\Big( h-\h(\rho) \Big) \\
			&=
			\partial_t 
			\Big( \rho\big( h-\h(\rho) \big) \Big)
			+
			\partial_x
			\Big( \rho v\big( h-\h(\rho) \big) \Big).
		\end{aligned}
	\end{equation*}
	Due to the third equation, i.e.~$
	\partial_t h
	=
	-v\partial_x h - \h'(\rho)\rho \partial_x v$, we obtain for the second equation the relation
	\begin{align*}
		\partial_t \Big(
		v+q(h)
		\Big)
		+
		v
		\partial_x \Big(
		v+q(h)
		\Big) 
		&=
		\partial_t v + q'(h) \Big[
		-v\partial_x h - \h'(\rho)\rho \partial_x v
		\Big]
		+
		v\partial_x v + vq'(h)\partial_x h\\
		&=
		\partial_t v
+		\Big(
		v-q'(h)\h'(\rho)\rho
		\Big) \partial_x v.
	\end{align*}

    \noindent
    Multiplying the first equation by~$v+q(h)$ and adding it to the second one yields
	\begin{align*}
		0&=\partial_t \Big(
		v+q(h)
		\Big)
		+
		v
		\partial_x \Big(
		v+q(h)
		\Big) 
		+
		\Big(
		v+q(h)
		\Big)
		\Big(
		\partial_t \rho+\partial_x (\rho v)
		\Big)\\
		&=
		\partial_t \Big(
		\rho\big(v+q(h)
		\Big)
		+
		\partial_x
		\Big(
		\rho v\big(v+q(h)
		\Big).
	\end{align*}
\hfill
\end{proof}

\noindent
Lemma~\ref{LemmaEntropy} states entropy-entropy flux pairs.
\medskip

\begin{lemma}\label{LemmaEntropy}
	The systems~\eqref{PiuS1}~--~\eqref{PiuS3} are endowed with the entropy-entropy flux pairs
	$$
	\entropy(\rho,v,h)
	=
	\rho F(\omega)
	\quad
	\text{and}
	\quad
	\entropyflux(\rho,v,h)
	=v \rho F(\omega)
	\quad\text{for}\quad
	\omega=v+q(h),
	$$
	where~$F$ is any convex function.
\end{lemma}
\medskip

\begin{proof}
	Due to the formulation~\eqref{PiuS2}, we have~$
	\partial_t \omega + v\partial_x \omega
	=
	0
	$. Hence, the entropy equation reads as
	\begin{equation*}
	\partial_t \entropy(\rho,v,h)
	+
	\partial_x \entropyflux(\rho,v,h)
	=
	\Big(
	\partial_t\rho
	+
	\partial_x(\rho v)
	\Big)
	F(\omega)
	+
	F'(\omega)\rho
	\Big(
	\partial_t \omega
	+
	v
	\partial_x \omega
	\Big)
	=0.
	\end{equation*}
    \hfill
\end{proof}

\noindent
Although two eigenvalues coincide, the Jacobian of the systems~\eqref{PiuS1}~--~\eqref{PiuS3} have a full set of eigenvectors and hence form a strongly hyperbolic system. 
In fact, systems that are endowed with an entropy-entropy flux pair can be symmetrized and hence are strongly hyperbolic~\cite[Th.~3.1]{Godlewski1998}.

Furthermore, we recall that a characteristic field, corresponding to an eigenvalue $\lambda_k$, is  genuinely nonlinear if the function~$\nabla_{\u}\lambda_k(\u)\cdot r_k(\u)$  never vanishes
and it is linearly degenerate if the equation~$\nabla_{\u}\lambda_k(\u)\cdot r_k(\u)=0$ holds for all states~$\u$. 
We observe the expressions~$
\nabla_{\u}\lambda_2(\u)\cdot r_2(\u)=0$, 
$\nabla_{\u}\lambda_3(\u)\cdot r_3(\u)=0
$ 
and the strict equality
$$
\nabla_{\u}\lambda_1(\u)\cdot r_1(\u)
=
-\rho q'(h)\big(\h'(\rho)+\rho \h''(\rho)\big)
-
q'(h)\h'(\rho)\rho
-q''(h)\h'(\rho)^2\rho^2<0
$$	
provided that the properties
$$\h'(\rho)>0,\quad q'(h)>0,\quad 
\h''(\rho)\geq 0,\quad q''(h)\geq 0$$
hold, which is  slightly more restrictive than the assumption~$
q'(h_i)h'(h_i)>0$
in Section~\ref{SecMicroPiu}, where the microscopic case is considered. 
Hence, the fields $k=2,3$ that correspond to the double eigenvalue~$\lambda_v(v)=v$, which travel faster then the first field, i.e.~$\lambda_1(\rho,v,h)<\lambda_v(v)$, 
form contacts.
Riemann invariants~$\R_k$ satisfy~$\nabla_\u\R_k(\u)\cdot\lambda_k =0$ and read as~$\R_1(\u)=\omega =v+q(h)$ 
and~$
\R_2(\u)=v$, 
$\R_3(\u)=v$. 

Therefore, the systems~\eqref{PiuS1}~--~\eqref{PiuS3} form  a Temple class system~\cite{Temple1982} such as the Aw-Rascle-Zhang traffic flow model~\cite[Sec.~2]{Aw:2000}. 
This is a desirable property for traffic flow models, since  oscillation cannot appear for such systems. They can  only be propagated~along contact discontinuities~\cite{DiPerna1983,Rascle1991}. 
More precisely, 
if a field is genuinely nonlinear, oscillations cannot remain for any positive time.

This means that  shock waves are stable and propagate in the correct direction, namely~backward relative to traffic flow, as observed in reality. 
Furthermore, the total entropy of the system is conserved, since the entropy inequality~$\partial_t \entropy + \partial_x \entropyflux =0$ is in fact an equality. In contrast to gas dynamics where information is lost over a shock wave, which yields mathematically an entropy inequality, the total entropy in traffic flow models is expected to be constant. 
For instance, a strict entropy inequality would mean that information about the driver dependent quantity~$h$ is lost across a shock~\cite{TechnicalReportEntropy,Huang2024AnEM}.

We note that oscillations are sometimes of interest, since they model stop-and-go waves. However, those arise from instabilities of small perturbations around steady states and not from a shock wave.

\section{Macroscopic traffic flow models with relaxation}
\label{SectionRelaxation} 
Typical macroscopic traffic flow models describe the density~$\rho$
and the mean velocity $v$ of vehicles. 
The natural assumption that the total mass is conserved leads to impose that the density $\rho$ satisfies the continuity equation
\begin{equation}\label{eq:LWR}
	\partial_t \rho + \partial_x(\rho v)=0
	\quad\text{with initial values} \quad \rho(0,x)=\rho_0(x).
\end{equation}
In first-order models, e.g.~the LWR model, the velocity $v=v(\rho)$ is given as a function of the density alone. 

\subsection{Relaxation in the Aw-Rascle-Zhang model}\label{sec:ARZ} 
Second-order models describe the velocity by an additional differential equation. In particular, we consider the inhomogeneous Aw-Rascle-Zhang model with relaxation
\begin{equation}\label{eq:arz_non_cons}
	\begin{cases}
		\begin{aligned}
			\partial_t \rho + \partial_x(\rho v)&=0,\\
			\partial_t \Big(v+P(\rho)\Big) + v \partial_x\Big(v+P(\rho)\Big)
			&=
			\frac{1}{\varepsilon}
			\Big(\Veq(\rho)-v\Big),  
		\end{aligned}
	\end{cases}
\end{equation}
which can be derived from the microscopic ODEs~\eqref{FtL} with relaxation
\begin{equation}\label{FtLBando}
	\begin{cases}
		\begin{aligned}
			\dot{x}_i&=v_i, \\
			\dot{v}_i&=c_\gamma \frac{\dot{x}_{i+1}-\dot{x}_{i}}{(x_{i+1}-x_i)^{\gamma+1}} + \frac{1}{\varepsilon}\Big( V_\text{eq}(\rho_i)-v_i\Big),
		\end{aligned}
	\end{cases}
\end{equation}
by scaling arguments or by considering the microscopic case as a numerical discretization~\cite{Aw:2002}. 
Here,  $P\; :\; \mathbb{R}^+\to\mathbb{R}^+$ is called  hesitation or traffic pressure \cite{fan2013comparative}. It is a smooth, strictly increasing  function of the density, whose form is determined by the Follow-the-Leader term, as showed in Section~\ref{SecARZmicro}. 
In order to establish a relation to the novel third-order model, we extend the commonly used notation~\eqref{eq:arz_non_cons} to~$P(\rho) = 
q\big(\HAR(\rho)\big)
$.

The relaxation term with parameter~${\varepsilon>0}$ on the right-hand side of equation~\eqref{eq:arz_non_cons} makes the drivers tend to a given equilibrium velocity~$\Veq(\rho)$. 
This is important, since the  homogeneous ARZ model without relaxation has no mechanism to
move drivers when they are initially  at rest. 
By introducing the variable $z=\rho\left(v+q\big(\HAR(\rho)\big)\right)$, 
the system~\eqref{eq:arz_non_cons} can be written in  conservative form as
\begin{equation}\label{eq:inhomo}
	\begin{cases}
		\begin{aligned}
			\partial_t \rho + \partial_x\Big(z-\rho q\big(\HAR(\rho)\big) \Big)&=0,\\
			\partial_t z + \partial_x\Big( \frac{z^2}{\rho}-z q\big(\HAR(\rho)\big) \Big)
			&=
			\frac{\rho}{\varepsilon}
			\Big(\Veq(\rho)-v(\rho,z)\Big)
		\end{aligned}
	\end{cases}
	\text{for}\ \ \
	v(\rho,z)=\frac{z}{\rho}-q\big(\HAR(\rho)\big). 
\end{equation}
Alternatively, the ARZ model can be written as
\begin{equation}\label{ARZ3}
	\begin{cases}
		\begin{aligned}
			\partial_t \rho + \partial_x(\rho v)&=0,\\
			\partial_t v+ \Big(v-q'\big(\HAR(\rho)\big)\HAR'(\rho)\rho\Big)\partial_x v
			&=\frac{1}{\varepsilon}
			\Big(\Veq(\rho)-v\Big).
		\end{aligned}
	\end{cases} 
\end{equation}
From the representation~\eqref{ARZ3} it becomes obvious that the characteristic speeds read as 
$$
\lambda_1^{\textup{ARZ}}(\rho,v)
=
v 
- 
q'\big(\HAR(\rho)\big)\HAR'(\rho)\rho
\quad\text{and}\quad
\lambda_2^{\textup{ARZ}}(\rho,v)
=
\lambda_v(v) \coloneqq
v.
$$
Hence, the ARZ model is strictly hyperbolic under the assumption $\rho>0$. 
The (local)  equilibrium velocity~$\Veq(\rho)$ satisfies in the limit~$\varepsilon\rightarrow 0$ the scalar conservation law
\begin{equation}\label{equilbriumEquation}
	\partial_t \rho
	+
	\partial_x 
	f_{\textup{eq}}(\rho)=0
	\quad\text{for}\quad
	f_{\textup{eq}}(\rho) = \rho \Veq(\rho)
	\quad\text{and}\quad
	f_{\textup{eq}}'(\rho) =  \Veq(\rho) + \rho \Veq'(\rho).
\end{equation}

\noindent
Stability requires that the full system propagates information faster than the local equilibrium~\cite{Jin1995}, i.e.~the~\textbf{sub-characteristic condition}
\begin{equation}\label{SC}\tag{\textup{SC}}
	\lambda_1^{\textup{ARZ}} \big( \rho,\Veq(\rho) \big) 
	\leq
	f_{\textup{eq}}'(\rho)
	\leq 
	\lambda_v \big( \Veq(\rho)  \big) 
	\quad\text{with}\quad \Veq'(\rho)<0
\end{equation}
is satisfied. 
The sub-characteristic condition holds if and only if the 
\textbf{first-order correction} 
that is obtained by the \textbf{Chapman-Enskog-expansion}
$$
v = \Veq(\rho) + \varepsilon v^{(1)}
+
\mathcal{O}\big(\varepsilon^2\big)
$$
leads to a dissipative advection-diffusion equation. For the deterministic ARZ model~\cite{Zhang:2002,Herty:2020}, this reads as
\begin{equation}\label{DI}\tag{\textup{DI}}
	\begin{aligned}
		&\partial_t \rho 
		+
		\partial_x f_{\textup{eq}}(\rho)
		=
		\varepsilon\partial_x
		\big(
		\mu(\rho)
		\partial_x \rho
		\big) \\
		&\text{with positive diffusion coefficient}\quad
		\mu(\rho)
		\coloneqq
		-\,
		\rho^2  \Veq'(\rho) \big( \Veq'(\rho) + h'(\rho) \big) \geq 0
	\end{aligned}
\end{equation}
provided that the inequality~$\Veq'(\rho) + h'(\rho) \geq 0$ holds. 
We emphasize that the equivalence of a sub-characteristic condition and a dissipative advection-diffusion equation, which has been established for general systems of size two~\cite[Th.~3.1]{Chen1994}, 
does~\emph{not} necessarily hold for hyperbolic relaxation problems of general size.

\subsection{Relaxation to the homogeneous Aw-Rascle-Zhang model}
Theorem~\ref{TheoremAwRascle1} states a similar statement for the TOM model in conservative form~\eqref{PiuS3}.\\

\begin{theorem}\label{TheoremAwRascle1}	
	Let an equilibrium~$
	\WAR(\rho)
	\coloneqq
	\rho
	\big(
	\HAR(\rho)
	-
	\h(\rho) \big)
	$
	be given and assume the function~$q$ is two-times differentiable in a small neighbourhood around~$\HAR(\rho)$. 
	Then, the first-order correction to the local equilibrium approximation of the relaxed model
	\begin{equation}\label{SystemRelaxAwRascle}
		\begin{cases}
			\begin{aligned}
				\partial_t \rho + \partial_x\Big(\rho v(\rho,z,w) \Big)&=0,\\
				\partial_t z+\partial_x 
				\Big(
				\frac{z^2}{\rho}
				-
				zq(h)
				\Big)
				&=0, \\
				\partial_t w +  \partial_x \Big(wv(\rho,z,w)
				\Big)
				&=
				-
				\frac{\WAR(\rho)-w}{\varepsilon}
			\end{aligned}
		\end{cases}
	\end{equation}	
	for 
	$\displaystyle v(\rho,z,w) = \frac{z}{\rho} - q(h)$ and 
	$\displaystyle w=\rho \big(
	h-\h(\rho)
	\big)$ 
	reads as
    $$
\begin{cases}
\begin{aligned}
\partial_t \rho + \partial_x \Big(
\rho v\big(\rho,z,\WAR(\rho)\big)
\Big)
&= 
	\varepsilon
	\partial_x
	\Big(\rho
\mu(\rho)
        \partial_x v\big(\rho,z,\WAR(\rho)\big)	\Big)
	+
	\mathcal{O}\big(
	\varepsilon^2
	\big), \\
	\partial_t z + 
	\partial_x
	\Big(
	\frac{z^2}{\rho}-z
	q\big(\HAR(\rho)\big)
	\Big) 
	&= 
	\varepsilon
	\partial_x \Big(z
\mu(\rho)
	\partial_x v\big(\rho,z,\WAR(\rho)\big)
	\Big)+
	\mathcal{O}\big(\varepsilon^2\big)
\end{aligned}
\end{cases}
$$
with~$\mu(\rho)\coloneqq	\rho	q'\big(\HAR(\rho)\big)
	\big(\h'(\rho)-\HAR'(\rho) \big)
$
	In particular, the diffusion coefficients~$\rho \mu(\rho)$ and~$z\mu(\rho)$ are positive provided that the inequality~$\HAR'(\rho)<\h'(\rho)$ holds.	
\end{theorem}
\medskip

\begin{proof}
	We apply the Chapman-Enskog expansion~$w = \WAR(\rho) + \varepsilon w^{(1)} + \mathcal{O}\big(\varepsilon^2\big)
	$ to the third equation of system~\eqref{SystemRelaxAwRascle} which yields also an expansion in terms of hesitation
	\begin{equation}\label{hChapman}
		h=\frac{w}{\rho}+\h(\rho)
		=
		\frac{\WAR(\rho)}{\rho}+\h(\rho)
		+
		\mathcal{O}(\varepsilon)
		=
		\HAR(\rho)
		+
		\mathcal{O}(\varepsilon)
		\ \
		\text{with}
		\ \
		\HAR(\rho)
		=
		\frac{\WAR(\rho)}{\rho}
		+\h(\rho).
	\end{equation}
	Taylor's theorem states for some~$\xi\in\big(h,\HAR(\rho) \big)$ or~$\xi\in\big(\HAR(\rho), h\big)$, respectively, which implies the bound~$ (h-\xi)=\mathcal{O}(\varepsilon)
	$, 
	the second-order representation
	\begin{equation}\label{Taylor}
		\begin{aligned}
			q(h)
			&=
			q\big(\HAR(\rho)\big)
			+
			\big(
			h-\HAR(\rho)
			\big)
			\Big[
			q'\big(
			\HAR(\rho)
			\big)
			+
			q''(\xi) (h-\xi)
			\Big] \\
			&=
			q\big(\HAR(\rho)\big)
			+
			\big(
			h-\HAR(\rho)
			\big)
			q'\big(\HAR(\rho)\big)
			+
	\mathcal{O}\big(\varepsilon^2\big).
		\end{aligned}
	\end{equation}
	The representation~\eqref{Taylor} and the expansion of~\eqref{hChapman} imply for the velocity the first-order representation
	$$
	v(\rho,z,w)
	=
	\frac{z}{\rho}
	-
	q(h)
	=
	\frac{z}{\rho}
	-
	q\big(
	\HAR(\rho)
	\big)
	+
	\mathcal{O}(\varepsilon)
	=
	v\big(
	\rho,z,\WAR(\rho)
	\big)
	+
	\mathcal{O}(\varepsilon).
	$$
	
	
	%
	%
	%

	\noindent
	Hence, the third equation satisfies
	\begin{alignat*}{8}
		&w^{(1)}
		&&=\partial_t \WAR(\rho) +  \partial_x \Big(\WAR(\rho) v\big(\rho,z,\WAR(\rho)\big)
		\Big)
		&&+\mathcal{O}(\varepsilon) \\
		& &&=
		-\WAR'(\rho)\partial_x\Big(\rho v\big(\rho,z,\WAR(\rho)\big) \Big) \\
	&	&&\quad +
		\WAR'(\rho)  v\big(\rho,z,\WAR(\rho)\big)\partial_x \rho + \WAR(\rho) \partial_x v\big(\rho,z,\WAR(\rho)\big)
		&&+
		\mathcal{O}(\varepsilon) \\
		& &&=-\WAR'(\rho)\rho \partial_x v\big(\rho,z,\WAR(\rho)\big)
		\; +
		\WAR(\rho) \partial_x v\big(\rho,z,\WAR(\rho)\big)
		&&+
		\mathcal{O}(\varepsilon)
	\end{alignat*}
	which leads to  the second-order expansions
	\begin{alignat*}{8}
&	w&&=
\	\WAR(\rho)
&&
&&	-
	\varepsilon
	\big(
	\WAR'(\rho)\rho-\WAR(\rho)
	\big)
	&&\partial_x v\big(\rho,z,\WAR(\rho)\big)
	&&+
\mathcal{O}\big(\varepsilon^2\big), \\
		& h&&= \ \
		\frac{w}{\rho}&&+\h(\rho)\\
& &&    		=
		\frac{\WAR(\rho)}{\rho}
        &&+\h(\rho)
	&&	-
		\varepsilon\bigg(
		\WAR'(\rho)-\frac{\WAR(\rho)}{\rho}
		\bigg)
        		&&\partial_x v\big(\rho,z,\WAR(\rho)\big)
		&&+\mathcal{O}\big(\varepsilon^2\big)\\
		& &&= \
		\HAR(\rho) && &&- \varepsilon \rho\big(
		\HAR'(\rho)
		-
		\h'(\rho)
		\big)
		&&\partial_x v\big(\rho,z,\WAR(\rho)\big)
		&&+
		\mathcal{O}\big(\varepsilon^2\big).
	\end{alignat*}
	Then, the second-order Taylor  representation~\eqref{Taylor} reads as
	\begin{equation}\label{proofSec43}
		q(h)=
		q\big(\HAR(\rho)\big)
		-
		\varepsilon \rho
		q'\big(\HAR(\rho)\big)
		\big(\HAR'(\rho)-\h'(\rho) \big)
		\partial_x v\big(\rho,z,\WAR(\rho)\big)		+		\mathcal{O}\big(\varepsilon^2\big)
	\end{equation}
	which implies for the second equation the expression
	
\begin{align*}
&	\partial_t z
	+
	\partial_x \Big(
	\frac{z}{\rho}
	-
	zq(h)
	\Big)\\
	=
&	\partial_t z
	+
	\partial_x \Big(
	\frac{z}{\rho}
	-
	zq\big(\HAR(\rho)\big)
	\Big)
	+
	\varepsilon
	\partial_x
	\Big(
	z\rho	q'\big(\HAR(\rho)\big)
	\big(\HAR'(\rho)-\h'(\rho) \big)
		\partial_x v\big(\rho,z,\WAR(\rho)\big)	\Big)
	+
	\mathcal{O}\big(
	\varepsilon^2
	\big).
\end{align*}
Furthermore, equation~\eqref{proofSec43} yields
\begin{equation*}    
v(\rho,z,w)
=\frac{z}{\rho}
-q(h) 
=\frac{z}{\rho}-
q\big(\HAR(\rho)\big)
+ \varepsilon \rho
q'\big(\HAR(\rho)\big)
 \big(
		\HAR'(\rho)
		-
		\h'(\rho)
		\big)
		\partial_x v\big(\rho,z,\WAR(\rho)\big)	+
\mathcal{O}\big(\varepsilon^2\big).
\end{equation*}
Hence, we have for the first equation the expression
\begin{align*}    
0&=
\partial_t \rho + \partial_x \Big(
\rho v\big(\rho,z,w\big)
\Big) \\
&=
\partial_t \rho + \partial_x \Big(
\rho v\big(\rho,z,\WAR(\rho)\big)
\Big) \\
&\quad +
	\varepsilon
	\partial_x
	\Big(\rho^2
		q'\big(\HAR(\rho)\big)
	\big(\HAR'(\rho)-\h'(\rho) \big)
		\partial_x v\big(\rho,z,\WAR(\rho)\big)	\Big)
	+
	\mathcal{O}\big(
	\varepsilon^2
	\big).
 \end{align*}   
Then, the claim follows from~$z=\rho\big(
v+q(h)\big)>0$.

\hfill	
\end{proof}

\medskip
\noindent
Under the assumption~$\HAR'(\rho)<\h'(\rho)$, when Theorem~\ref{TheoremAwRascle1} guarantees positive diffusion coefficients~$\rho \mu(\rho)>0$ and~$z \mu(\rho)>0$, we have the inequality 
\begin{equation}\label{subCharacteristicPiu}
	\begin{aligned}
		\lambda_1\big(\rho,v,\HAR(\rho) \big) 
		&= 
		v-
		q'\big(
		\HAR(\rho)
		\big) \,
		\h'(\rho)\rho \\
		&<
		v-
		q'\big(
		\HAR(\rho)
		\big)
		\HAR'(\rho)\rho 
		=
		\lambda_1^\textup{ARZ}\big(\rho,v \big) 
		\leq
		v
		= 
		\lambda_v(v)
	\end{aligned}
\end{equation}
with velocity~$
v
=
v\big(\rho,z,\WAR(\rho)\big)
$. 
The property~\eqref{subCharacteristicPiu} in turn ensures the important sub-characteristic condition for the relaxation problem  of size three in Theorem~\ref{TheoremAwRascle1}. Namely, 
the equilibrium characteristics~$\lambda_1^\textup{ARZ}\big(\rho,v \big)$ and 
$\lambda_2^\textup{ARZ}\big(\rho,v \big)=\lambda_v(v)$, which coincide with those of the Aw-Rascle-Zhang model, are within the range of the frozen characteristics $\lambda_1\big(\rho,v,\HAR(\rho) \big) $ 
and $\lambda_2\big(\rho,v,\HAR(\rho) \big)
=
\lambda_3\big(\rho,v,\HAR(\rho) \big)
=
\lambda_v(v) = v
$ 
if and only if the advection-diffusion equation is dissipative.

\subsection{Analysis of  the frozen characteristics}
Next, we examine the behaviour close to the 
frozen characteristics~${
	\lambda_1 \big(\rho,v,\HAR(\rho)\big)
}$. To this end, the relaxation is stated with respect to the non-conservative form~\eqref{PiuS1}, 
which yields 
a relaxation term that is of Hamilton-Jacobi-type (HJ), where the solution is understood in terms of viscosity solutions~\cite{Crandall1983,Caselles1992}. 

\medskip

\begin{theorem}\label{TheoremAwRascle2}
	Let an equilibrium~$
	\HAR(\rho)
	$ 
	be given and assume the function~$q$ is three-times differentiable in a small neighbourhood around~$\HAR(\rho)$. 
	Then, the first-order correction to the local equilibrium approximation of the relaxation
	\begin{equation}\label{RelaxationHyperbolic}
		\begin{cases}
			\begin{aligned}
				\partial_t \rho + \partial_x(\rho v)&=0,\\
				\partial_t v+ \Big(v-q'(h)\h'(\rho)\rho\Big)\partial_x v
				&= 0, \\
				\partial_t
				\Big( h-\h(\rho) \Big) 
				+
				v
				\partial_x
				\Big( h-\h(\rho) \Big)
				&=-
				\frac{\HAR(\rho)-h}{\varepsilon}
			\end{aligned}
		\end{cases}
	\end{equation}
	reads as
	\begin{equation*}
		\begin{cases}
			\begin{aligned}
				\partial_t \rho + \partial_x(\rho v)&=0,\\
				\partial_t v +
				\Big(
				v-q'\big(\HAR(\rho)\big)\h'(\rho)\rho
				\Big)	\partial_x v
				&=
-				\varepsilon
\rho \gamma(\rho)
(\partial_x v)^2
				+\mathcal{O}\big(\varepsilon^2\big)	\end{aligned}
		\end{cases}
	\end{equation*}
for 
$
	\gamma(\rho)\coloneqq
	q''\big(
	\HAR(\rho)
	\big)
	\h'(\rho)
	\rho
	\big(\HAR'(\rho)-\h'(\rho) \big).
$  
	In terms of viscosity solutions, it is equivalent to a hyperbolic system with eigenvalues
	\begin{align*}	
	\lambda_1^{\textup{HJ}} (\rho,v) &= \Big(
	v-q'\big(\HAR(\rho)\big)\h'(\rho)\rho
	\Big)-\varepsilon 
    \gamma(\rho)
	\dot{\rho}, \\
	\lambda_2^{\textup{HJ}}(\rho,v)
	&=v.
	\end{align*}

	
\end{theorem}

\begin{proof}	
	We apply the Chapman-Enskog expansion~$h=\HAR(\rho)+\varepsilon h^{(1)}+\mathcal{O}\big(\varepsilon^2\big)$	to the third equation of system~\eqref{RelaxationHyperbolic}, which yields
	\begin{alignat*}{8}
		&	h^{(1)}
		&&=
		\partial_t
		&&\Big( \HAR(\rho)&&-\h(\rho) &&\Big) 
		+
		v
		\partial_x
		\Big( \HAR(\rho)-\h(\rho) \Big)
		&&+\mathcal{O}(\varepsilon) \\
		& &&=
		&&\Big(\HAR'(\rho)&&-\h'(\rho) &&\Big)
		\Big(
		\partial_t \rho + v \partial_x \rho
		\Big)	
		&&+\mathcal{O}(\varepsilon) \\
		& &&=-
		&&\Big(\HAR'(\rho)&&-\h'(\rho) &&\Big)
		\rho\, \partial_x v
		&&+\mathcal{O}(\varepsilon).
	\end{alignat*}
	Hence, we have the expansion~$h=\HAR(\rho)
	-
	\varepsilon \rho\big(\HAR'(\rho)-\h'(\rho) \big)
	\partial_x v+\mathcal{O}\big(\varepsilon^2\big)$.  Taylor's theorem states for some~$\xi\in\big(h,\HAR(\rho) \big)$ or~$\xi\in\big(\HAR(\rho), h\big)$, respectively, which implies the bound~$ (h-\xi)=\mathcal{O}(\varepsilon)
	$, 
	the approximation
	\begin{align*}
		q'(h)
		&=
		q'\big(\HAR(\rho))
		+
		\big(
		h-\HAR(\rho)
		\big)
		\Big[
		q''\big(
		\HAR(\rho)
		\big)
		+
		q^{(3)}(\xi) (h-\xi)
		\Big] \\
		&=
		q'\big(\HAR(\rho))
		-
		\varepsilon
		q''\big(\HAR(\rho)\big)
		\rho\big(\HAR'(\rho)-\h'(\rho) \big)
		\partial_x v
		+
		\mathcal{O}\big(\varepsilon^2\big).
	\end{align*}
	Then, the second equation of  system~\eqref{RelaxationHyperbolic} reads as
	\begin{align}
		0&=	\partial_t v + \Big(v-q'(h)\h'(\rho)\rho\Big)\partial_x v \label{HJ}\\
		&	=
		\partial_t v +
		\Big(
		v-
		q'\big(\HAR(\rho)\big)\h'(\rho)\rho
		\Big)	\partial_x v
		+
		\varepsilon
        \rho
        \gamma(\rho)(\partial_x v)^2
		+\mathcal{O}\big(\varepsilon^2\big).\nonumber
	\end{align}
	We introduce the auxiliary variables 
	$\alpha\coloneqq \partial_x v
    $ and
    $
	\beta(\rho)
	\coloneqq
	q'\big(
	\HAR(\rho)
    \big)
	\h'(\rho)\rho
	$. 
	Differentiating equation~\eqref{HJ} with respect to space~${x\in\mathbb{R}}$ yields
	\begin{equation}\label{HJ1}
\begin{aligned}
0&=		\partial_t \alpha +
		\Big(
		v-\beta(\rho)
		\Big)	\partial_x \alpha \\
& \quad		+
		\alpha\partial_x v- \alpha
		\beta'(\rho)\partial_x\rho
		+\varepsilon\Big(
		\alpha^2 \partial_\rho \big( \rho\gamma(\rho) \big)\partial_x \rho 
		+2\rho \gamma(\rho)\alpha \partial_x \alpha \Big)
        +
        \mathcal{O}\big(\varepsilon^2\big).
\end{aligned}
	\end{equation}
	Neglecting terms of order $\mathcal{O}\big(\varepsilon^2\big)$, the continuity equation~$\partial_t \rho + \partial_x(\rho v) =0$, equation~\eqref{HJ} and~\eqref{HJ1} 
	form altogether the hyperbolic quasilinear form 
    $
    \partial_t \u + Q(\u) \partial_x \u=0 
    $ 
    with unknowns~$
    \u= (\rho,v,\alpha)^\T
    $ and matrix function
	$$
Q(\u)=
    	\begin{pmatrix}
		v & \rho & 0 \\
		0	& v-\beta(\rho) + \varepsilon \rho \alpha \gamma(\rho) & 0 \\
		-\alpha \beta'(\rho) 
		+
		\varepsilon
		\alpha^2 \partial_\rho \big( \rho\gamma(\rho) \big)
		& \alpha
		& v-\beta(\rho)+\varepsilon 2\rho \alpha \gamma(\rho)
	\end{pmatrix}.
	$$
	Its characteristic speeds read as
	$$
	\lambda_1^{\textup{HJ}}(\rho,v) = \big(
	v-\beta(\rho)
	\big)+
	\varepsilon \rho \alpha \gamma(\rho), \ \
	\lambda_2^{\textup{HJ}}(\rho,v)
	=v    
	\ \ \text{and}\ \  
	\widetilde{\lambda_3^{\textup{HJ}}}(\rho,v) = \big(
	v-\beta(\rho)
	\big)+
	\varepsilon 2\rho \alpha \gamma(\rho).
	$$
	The claim follows from the relation~$\dot{\rho} = -\rho \alpha$. 

        \hfill	
\end{proof}

\medskip

\noindent
Theorem~\ref{TheoremAwRascle2} allows to properly choose functions~$\HAR(\rho)$, $\h(\rho)$ and $q(h)$ that account for different driver's  attitudes depending on de- and acceleration. Namely, the frozen characteristic~$\lambda_1\big(\rho,v,\HAR(\rho) \big) $ 
states how fast information travels backward as a function of the density. This speed depends in the TOM model also on acceleration and deceleration as the relation
\begin{equation*}
	\lambda_1^{\textup{HJ}}(\rho,v) = \lambda_1\big(\rho,v,\HAR(\rho) \big)
	-
	\varepsilon \gamma(\rho) \dot{\rho} 
\end{equation*} 
for 	$\gamma(\rho)
	= 
	q''\big(
	\HAR(\rho)
	\big)
	\h'(\rho)\rho
	\big(\HAR'(\rho)-\h'(\rho) \big)
$
shows. This allows to model the in the real world observable features
\begin{alignat*}{8}
	&\lambda_1^{\textup{HJ}}(\rho,v)
	&&   \leq
	\lambda_1\big(\rho,v,\HAR(\rho) \big) 
	&&\quad \text{if drivers accelerate,}\\
	&\lambda_1^{\textup{HJ}}(\rho,v)
	&&   \geq
	\lambda_1\big(\rho,v,\HAR(\rho) \big) 
	&&\quad \text{if drivers decelerate.}
\end{alignat*}

\begin{remark}
	We note the similarity and the main difference between the proofs of Theorem~\ref{TheoremAwRascle1}
	and Theorem~\ref{TheoremAwRascle2}. 
	Starting from the relaxation to the third equation, we obtain expressions for the first-order Chapman-Enskog expansion in terms of
	\begin{alignat*}{8}
		&    h&&=\HAR(\rho)
		&&	-
		\varepsilon \rho\big(\HAR'(\rho)-\h'(\rho) \big)
		&&	\partial_x v
		&&+
		\mathcal{O}\big(\varepsilon^2\big),\\
		&	w&&=
		\WAR(\rho)
		&&	-
		\varepsilon
		\big[
		\WAR'(\rho)\rho-\WAR(\rho)
		\big]
		&&	\partial_x v\big(\rho,z,\WAR(\rho)\big)
		&&	+
		\mathcal{O}\big(\varepsilon^2\big).
	\end{alignat*}
	In fact, both expressions are equivalent up to second order although they are stated with respect to different variables. 
	Then, the Chapman-Enskog-type expansions are inserted into the second equation for the velocity. 
	Theorem~\ref{TheoremAwRascle2} is based on a non-conservative form of the second equation. This leads to a relaxation term that is of Hamilton-Jacobi-type. 
	In contrast, Theorem~\ref{TheoremAwRascle1} makes use of a conservative form which leads to an advection-diffusion equation. In this case, we obtain a similar result as in Section~\ref{sec:ARZ}, where the relaxation of the Aw-Rascle-Zhang model to the first-order LWR model is also related to an advection-diffusion equation. An explanation is given in~\cite{Herty:2020}, 
	where a Chapman-Enskog-type expansion is  inserted into  a conservative form, namely the continuity equation.

\end{remark}

\subsection{Relaxation to a prescribed velocity of a multiclass system}
We examine in this section the relation to traffic flow models of the type
\begin{equation}\label{FixedVelocity}
	\partial_t
	\begin{pmatrix}
		\rho \\ \rho \omegaTilde
	\end{pmatrix}
	+\partial_x
	\begin{pmatrix}
		\Vtilde(\rho,\omegaTilde)	\rho \\ \Vtilde(\rho,\omegaTilde) \rho \omegaTilde
	\end{pmatrix}
	=0
	\quad
	\text{or equivalently}
	\quad
	\begin{cases}
		\begin{aligned}
			\partial_t \rho + \partial_x\big(\rho \Vtilde(\rho,\omegaTilde) \big)&=0, \\
			\partial_t \omegaTilde + \Vtilde(\rho,\omegaTilde) \partial_x \omegaTilde&= 0,
		\end{aligned}
	\end{cases}
\end{equation}
where~$\omegaTilde$ is a driver-dependent quantity. Such systems generate a  
family of velocity curves~$\Vtilde(\cdot,\omegaTilde)$. 
As explained in~\cite[Sec.~2.3]{fan2013comparative}, 
the ARZ model is of this form. Furthermore, the  generalization of the LWR model can be explained by the choice~$\omegaTilde=\omega=v+P(\rho)$, where this driver-dependent quantity  moves with the vehicles and 
influences the family of velocity curves~$
\Vtilde(\rho,\omega)
=
\omega
-P(\rho) 
$. \\

\noindent
Another traffic flow model of the form~\eqref{FixedVelocity} has been recently introduced to account for hysteresis~\cite{Corli:2019}. 
Since the speed-density relations corresponding to acceleration and deceleration 
are usually different, the function
$$
\h(\rho) \coloneqq
\chi_{A}(\rho) \hA(\rho)
+
\chi_{D}(\rho) \hD(\rho)
$$ 
is made dependent on the regions
$$
A(\rho) \coloneqq
\Big\{
(t,x) \in \mathbb{R}\times \mathbb{R}^+
\ \Big| \
\dot{\rho}(t,x)<0
\Big\}
\quad\text{and}\quad
D(\rho) \coloneqq
\Big\{
(t,x) \in \mathbb{R}\times \mathbb{R}^+
\ \Big| \
\dot{\rho}(t,x)>0
\Big\},
$$
where cars accelerate or slow down. 
Then, a specific~velocity~$\vCorli(\rho,h)$ is derived that depends on hysteresis~$h$ to account for the multi-valued speed-density relations. By defining the term
\begin{equation*}
	\HCorli(\rho,v)
	\coloneqq
	\chi_{A}(\rho) \Big(
	\partial_t\hA(\rho)
	+
	v \partial_x\hA(\rho)
	\Big)
	+
	\chi_{D}(\rho) \Big(
	\partial_t\hD(\rho)
	+
	v \partial_x\hD(\rho)
	\Big),
\end{equation*}
the model in~\cite{Corli:2019} can be written as
\begin{equation}\label{CorliLWR}
	\begin{cases}
		\begin{aligned}
			\partial_t \rho + \partial_x\big(\rho \vCorli(\rho,h) \big)&=0, \\
			\partial_t h + \vCorli(\rho,h) \partial_xh&= \HCorli\big(\rho,\vCorli(\rho,h)\big).
		\end{aligned}
	\end{cases}
\end{equation}
Therefore, the system~\eqref{CorliLWR} is of the form  given in equation~\eqref{FixedVelocity}. 
Here, we have
\begin{alignat*}{8}
&	\HCorli\big(\rho,\vCorli(\rho,h)\big)
	&&= &&
	\hCorli'(\rho) \Big[ -
	\partial_x \big( \rho \vCorli(\rho,h) \big) 
	+ 
	\vCorli(\rho,h) \partial_x \rho \Big] \\
&	&&= -&&
	\hCorli'(\rho) \Big[ 
	\rho \partial_\rho \vCorli(\rho,h)\partial_x \rho
	+
	\rho \partial_h \vCorli(\rho,h)\partial_x h \Big].
\end{alignat*}
which yields the quasilinear form~$
\partial_t \u + Q(\u) \partial_x \u=0 
$ 
for the unknowns~$\u=( \rho,h )^\T$ and the matrix function
$$
Q(\rho,h)=
\begin{pmatrix}
	\vCorli(\rho,h)
	+
	\rho \partial_\rho \vCorli(\rho,h)
	&
	\rho \partial_h \vCorli(\rho) \\
	\hCorli'(\rho) 
	\rho \partial_\rho \vCorli(\rho,h)
	&
	\hCorli'(\rho)
	\rho \partial_h \vCorli(\rho,h) 
	+\vCorli(\rho,h)
\end{pmatrix}.
$$
Hence, the characteristic speeds are
$$
\lambda_1^C(\rho,h)
=
\vCorli(\rho,h)
+
\rho\big[
\partial_\rho \vCorli(\rho,h)
+
\hCorli'(\rho)
\partial_h \vCorli(\rho,h)\big]
\quad
\text{and}
\quad
\lambda_2^C(\rho,h)
=
\vCorli(\rho,h).
$$

\noindent
While the perspective of~\cite{fan2013comparative} is more from a parameterized multiclass system, where driver dependent properties specify a given velocity curve, the approach of~\cite{Corli:2019,Corli:2023} is more related to the idea of the here proposed model, namely a closure problem. Inspired by similarly interpreting $h$  as hysteresis, we prove the following theorem. 

\medskip

\begin{theorem}\label{TheoremRelaxationCorli}
	The first-order correction to the local equilibrium approximation of the relaxed model
	\begin{equation}\label{Corli}
		\begin{cases}
			\begin{aligned}
				\partial_t \rho + \partial_x(\rho v)&=0,\\
				\partial_t v+ \Big(v-q'(h)\h'(\rho)\rho\Big)\partial_x v
				&= \frac{ \vCorli(\rho,h)-v }{\varepsilon} , \\
				\partial_t w +  \partial_x (wv
				)
				&=
				0		\end{aligned}
		\end{cases}
	\end{equation}	
	for 
	$w
	=
	\rho\big(
	h-\h(\rho) \big)
	$   
	reads as
	\begin{equation}\label{RelaxationCorli}
		\partial_t
		\begin{pmatrix}
			\rho \\ w
		\end{pmatrix}
		+\partial_x
		\begin{pmatrix}
			\vCorli(\rho,h)	\rho \\ \vCorli(\rho,h) w 
		\end{pmatrix}
		=
		-
		\varepsilon
		\partial_x 
		\left(
		\mu_C(\rho,h)
		\partial_x
		\begin{pmatrix}
			\rho \\ h
		\end{pmatrix}
		\right)
		+\mathcal{O}\big(\varepsilon^2\big)
	\end{equation}	
	with diffusion matrix
	\begin{align*}
		&\muCorli(\rho,h)
		\coloneqq
		\Big(
		\lambda_1^C(\rho,h)
		-\big(\vCorli(\rho,h)
		-q'(h)\h'(\rho)\rho
		\big)\Big)
		\begin{pmatrix}
			\rho	\partial_\rho \vCorli(\rho,h) 
			& 
			\rho	\partial_h \vCorli(\rho,h) \\
			w	\partial_\rho \vCorli(\rho,h) 
			& 
			w	\partial_h \vCorli(\rho,h) 	
		\end{pmatrix}\\
		&\text{for}\quad
		\lambda_1^C(\rho,h)
		=
		\vCorli(\rho,h)
		+
		\rho\big[
		\partial_\rho \vCorli(\rho,h)
		+
		\h'(\rho)
		\partial_h \vCorli(\rho,h)\big].
	\end{align*}
	In particular, if the relation
	$$
	q'(h)\h'(\rho)\rho
	=
	\vCorli(\rho,h)-\lambda_1^C(\rho,h)
	=
	\rho\big[
	\partial_\rho \vCorli(\rho,h)
	+
	\h'(\rho)
	\partial_h \vCorli(\rho,h)\big]
	$$
	holds, the equations~\eqref{Corli} and~\eqref{RelaxationCorli} coincide up to terms of order~$\mathcal{O}\big(\varepsilon^2\big)$.
	
\end{theorem}

\medskip

\begin{proof}
	We apply a Chapman-Enskog expansion
	$$
	v = \vCorli(\rho,h) + \varepsilon \v1
	+
	\mathcal{O}\big( \varepsilon^2 \big).
	$$	
	Due to~$
	\partial_t \h(\rho)+v\partial_x \h(\rho) = -\h'(\rho)\rho \partial_x v
	$, the first and third equation imply
	\begin{alignat*}{8}
		&\partial_t \rho 
		&&=-
		\partial_x \big( \rho \vCorli(\rho,h) \big)
		&&+\mathcal{O}(\varepsilon) \\
		& &&=
		-
		\vCorli(\rho,h) 
		\partial_x  \rho
		-
		\rho \Big( 
		\partial_\rho \vCorli(\rho,h)\partial_x \rho
		+
		\partial_h \vCorli(\rho,h)\partial_x h
		\Big)
		&&+\mathcal{O}(\varepsilon), \\
		&\partial_t h &&= 
        -\vCorli(\rho,h) \partial_x h
        - \h'(\rho) \rho \partial_x \vCorli(\rho,h) 
		 &&+\mathcal{O}(\varepsilon)\\
		& &&=
        -\vCorli(\rho,h) \partial_x h 
        -
		\h'(\rho) \rho
		\Big(\partial_\rho \vCorli(\rho,h)\partial_x \rho
		+
		\partial_h \vCorli(\rho,h)\partial_x h
		\Big)
		&&+\mathcal{O}(\varepsilon).
	\end{alignat*}
	Hence, we have asymptotically the expression
	\begin{alignat}{8}
		&	\partial_t \vCorli(\rho,h) + \Big(\vCorli(\rho,h)-q'(h)\h'(\rho)\rho\Big) \partial_x \vCorli(\rho,h) \label{ProofRelaxationCorli1}\\
		&	
		=	
		\partial_\rho \vCorli(\rho,h) \partial_t \rho 
		+ 
		\partial_h \vCorli(\rho,h) \partial_t h \nonumber\\
	&\quad	+
		\Big(\vCorli(\rho,h)-q'(h)\h'(\rho)\rho\Big) \Big(
		\partial_\rho \vCorli(\rho,h) \partial_x \rho 
		+ 
		\partial_h \vCorli(\rho,h) \partial_x h 
		\Big)  \nonumber\\
		&
		=-	 
		\partial_\rho \vCorli(\rho,h) \vCorli(\rho,h)
		\partial_x \rho
		-
		\rho\partial_\rho \vCorli(\rho,h) \Big( 
		\partial_\rho \vCorli(\rho,h)\partial_x \rho
		+
		\partial_h \vCorli(\rho,h)\partial_x h
		\Big)  \nonumber\\
		& 
		\quad-
		\partial_h \vCorli(\rho,h)
		\h'(\rho) \rho
		\Big(\partial_\rho \vCorli(\rho,h)\partial_x \rho
		+
		\partial_h \vCorli(\rho,h)\partial_x h \Big)
		-
		\partial_h \vCorli(\rho,h)
		\vCorli(\rho,h) \partial_x h   \nonumber\\ 
		& \quad +  
		\Big(\vCorli(\rho,h)-q'(h)\h'(\rho)\rho\Big) \Big(
		\partial_\rho \vCorli(\rho,h) \partial_x \rho 
		+ 
		\partial_h \vCorli(\rho,h) \partial_x h  \Big) &&+\mathcal{O}(\varepsilon) \nonumber\\
		& = - \Big(  \partial_\rho \vCorli(\rho,h)
		\partial_x \rho
		+
		\partial_h \vCorli(\rho,h) 
		\partial_x h
		\Big)
		\Big(
		\lambda_1^C(\rho,h)
		-\big(\vCorli(\rho,h)
		-q'(h)\h'(\rho)\rho
		\big)\Big)
        && + \mathcal{O}(\varepsilon), \nonumber
	\end{alignat}
where the last equality follows from the relation
$$
		\Big(
		\partial_\rho \vCorli(\rho,h) \rho
		+
		\partial_h \vCorli(\rho,h) \h'(\rho)\rho
		+q'(h)\h'(\rho)\rho\Big)
=
		\Big(
		\lambda_1^C(\rho,h)
		-\big(\vCorli(\rho,h)
		-q'(h)\h'(\rho)\rho
		\big)\Big).
$$
	Furthermore, we have
	\begin{alignat}{8}
		&\v1 &&= 
		\frac{v-\vCorli(\rho,h)}{\varepsilon}
		&& && &&+\mathcal{O}(\varepsilon) \nonumber\\
		& &&=
		-\partial_t v &&- \Big(v&&-q'(h)\h'(\rho)\rho\Big) \partial_xv
		&&+\mathcal{O}(\varepsilon) \nonumber\\
		& &&=-
		\partial_t \vCorli(\rho,h) &&+ \Big(\vCorli(\rho,h)&&-q'(h)\h'(\rho)\rho\Big) \partial_x \vCorli(\rho,h)
		&&+\mathcal{O}(\varepsilon).
		\label{ProofRelaxationCorli2}
	\end{alignat}
	
	\noindent
	The equations~\eqref{ProofRelaxationCorli1}, \eqref{ProofRelaxationCorli2} and  the expressions
	\begin{alignat*}{8}
		&	\partial_t \rho &&+ \partial_x \Big(
		\vCorli(\rho,h) \rho
		\Big)
		&&=
		-\varepsilon \partial_x \Big( \rho \v1 \Big) &&+ \mathcal{O}\big(\varepsilon^2\big), \\
		&	\partial_t w &&+ \partial_x \Big(
		\vCorli(\rho,h) w
		\Big)
		&&=
		-\varepsilon \partial_x \Big( w \v1 \Big) &&+ \mathcal{O}\big(\varepsilon^2\big) 
	\end{alignat*}	
	yield the claimed diffusion matrix.
    
\hfill	
\end{proof}

\medskip

\noindent
Theorem~\ref{TheoremRelaxationCorli} 
states a possible choice of the functions~$
q
$ and~$\h$ such that the novel proposed model leads to results that are justified in the recent publication~\cite{Corli:2019,Corli:2023}. Namely, we have a reasonable choice  if the relation $
q'(h)\h'(\rho)\rho
\approx
\vCorli(\rho,h)-\lambda_1^C(\rho,h)
$ 
holds approximately, 
where~$\vCorli(\cdot,h)$ is the family of velocity curves from~\cite{Corli:2019} that takes hysteresis effects into account, 
which can model stop-and-go waves.

\section{Discussion of the Riemann problem}\label{SectionNumerics}
In this section we discuss the solution to a Riemann problem with constant left~$\u_\ell$
and right~$\u_r$
state. 
As explained in Section~\ref{SectionMacroscopicPiu}, 
the TOM model~\eqref{PiuS1}~--~\eqref{PiuS3} forms a Temple class system where the fields $k=2,3$, which correspond to the double eigenvalue~$\lambda_v(v)=v$  and which travel faster than the first wave,
form contacts. 
Since the first field is genuinely nonlinear, it can  either produce a rarefaction or a shock wave, which is connected by an intermediate state 
$\bar{\u}=
(
\bar{\rho},\bar{v},\bar{h}
)^\T$ 
with the contact. 

Since the velocity is a Riemann invariant of the second and third field, i.e.~${
\R_2(\u)=v}$, 
$\R_3(\u)=v$, 
it remains constant over a contact wave and hence the velocity of the intermediate state satisfies~$\bar{v}=v_r$. The third unknown~$\bar{h}$ is determined by the relation~$v_\ell+q(h_\ell)
=\R_1(\u_\ell)=\R_1(\bar{\u})=\bar{v}+q(\bar{h})
=
v_r+q(\bar{h})$.

The solution corresponding to a rarefaction wave travels according to an integral curve~$\mathcal{I}(\cdot;\u) $ of the first field, which   satisfies the ODE~$
\partial_\sigma \mathcal{I}_1(\sigma;\u)
=
r_1 \big(
\mathcal{I}(\sigma;\u)
\big)$ for $
\mathcal{I}(0;\u)
=
\u
$ 
provided that it holds~${
	\mathcal{I}( \sigma^*;\u_\ell)
    =
  ( \mathcal{I}_1,\mathcal{I}_2,\mathcal{I}_3)^\T( \sigma^*;\u_\ell)=
    \bar{\u} 
}$ for some~${\sigma^*\geq 0}$. 
Since the intermediate states~$\bar{v}$ and~$\bar{h}$ are given due to the foregoing discussion, the value~$\sigma^*$ is obtained by the equations~$\mathcal{I}_2( \sigma^*;\u_\ell) = \bar{v}$ and~$\mathcal{I}_3( \sigma^*;\u_\ell) = \bar{h}$. Finally, the density in the intermediate state is given by~$\bar{\rho}=\mathcal{I}_1( \sigma^*;\u_\ell)$. 
According to~{\cite[Sec.~5.1]{BRESSAN} and an entropy admissible weak solution is then given by

$$
\u(t,x)
\coloneqq
\begin{cases}
\u_\ell
& \text{for \ \ } x<t \lambda_1(\u_\ell), \\
\mathcal{I}(\sigma;\u_\ell)
& \text{for \ \ } x\in \big( t\lambda_1(\u_\ell),t\lambda_1(\bar{\u})  \big) 
\text{ \ \ and \ \  } \nicefrac{x}{t}   =\lambda_1\big(
\mathcal{I}(\sigma;\u_\ell)
\big), \\
\bar{\u}
& \text{for \ \ } x\in \big( t\lambda_1(\bar{\u}),tv_r  \big), \\
\u_r
& \text{for \ \ } x>t v_r.
\end{cases}
$$

Figures~\ref{Figure1}~--~\ref{Figure3} show the solution to the Riemann problem with initial values
$
\rho_\ell
=0.8
$, 
$
\rho_r
=0.1
$, 
$
v_\ell
=0.2
$, 
$
v_r
=0.5
$,~$h_\ell=\rho_\ell^{\nicefrac{3}{2}}$,~$h_r=\rho_r^{\nicefrac{3}{2}}$ and initial jump at~$x_0=0$ 
for different choices of the functions~$q(h)$,~$\h(\rho)$. 
More precisely, the panels~(i)~--~(iii) state the density~$\rho$ (blue), velocity~$v$ (red) and the driver-dependent quantity~$h$ (green) with respect to the left $y$-axis. 
The characteristic speeds~$\lambda_1(\u)$ are plotted with respect to the scale of the right $y$-axis. 
Since they are increasing, characteristics do not cross and a rarefaction wave occurs and not a shock. 
Panel (vi) states the difference 
\begin{align*}
    &
   \Delta v(t,x) \coloneqq 
    v(t,x)  
-
v^{\textup{AR}}(t,x)
\quad\text{for}\quad
v^{\textup{AR}}(t,x) \coloneqq v_\ell + \frac{x - t\lambda_1(\u_\ell)    }
{t\lambda_1(\bar{\u}) - t\lambda_1(\u_\ell)   }
\big(\bar{v} - v_\ell\big) \\
&\text{and}\quad
x\in \big( t\lambda_1(\u_\ell),t\lambda_1(\bar{\u})  \big).
\end{align*}
More precisely, the velocity that is described by the classical Aw-Rascle-Zhang model has a linear increase in the rarefaction wave, which is a known feature of this model~\cite{Aw:2002}. 
We observe from the representation~\eqref{PiuS1} that the TOM model coincides with the Aw-Rascle-Zhang model in the case~$q(h)=h$. Therefore,  the fourth panel of Figure~\ref{Figure2} shows  a linear increase of the velocity over a rarefaction wave, i.e.~the difference~$\Delta v(t,x)$ is zero. Apart from this case, however, the acceleration of drivers is no more constant, which  overcomes this point of criticism concerning the Aw-Rascle-Zhang model.

\begin{figure}[H]

\begin{minipage}{0.49\textwidth}
	
	\begin{center}	
		 \textbf{(i)}\ \  
$\boldsymbol{q(h)=h^{\nicefrac{1}{2}}}$, 
$\boldsymbol{\h(\rho)=\rho^{\nicefrac{1}{2}}}$
		\vspace{-1mm}			
		\scalebox{1}{\includegraphics[width=\linewidth]{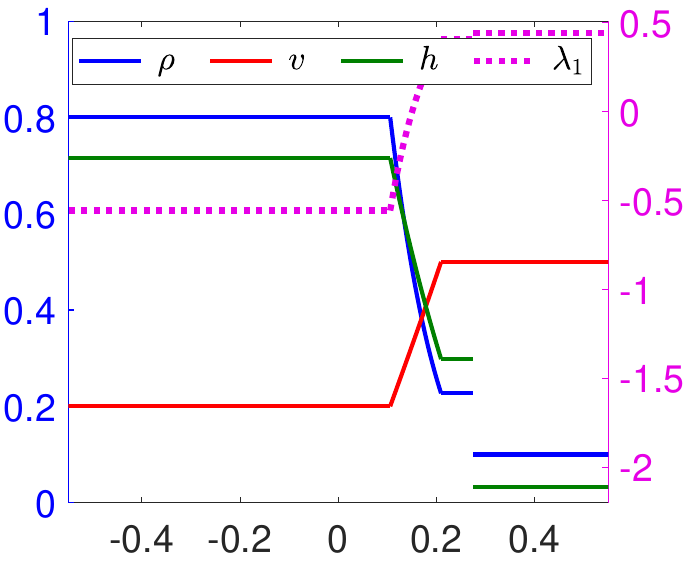}}

	\end{center}	
\end{minipage}
\hfil
\begin{minipage}{0.49\textwidth}
	
	\vspace{-2mm}
	\begin{center}
		 \textbf{(ii)}\ \  
$\boldsymbol{q(h)=h^{\nicefrac{1}{2}}}$, 
$\boldsymbol{\h(\rho)=\rho}$
\vspace{-1mm}			
		\scalebox{1}{\includegraphics[width=\linewidth]{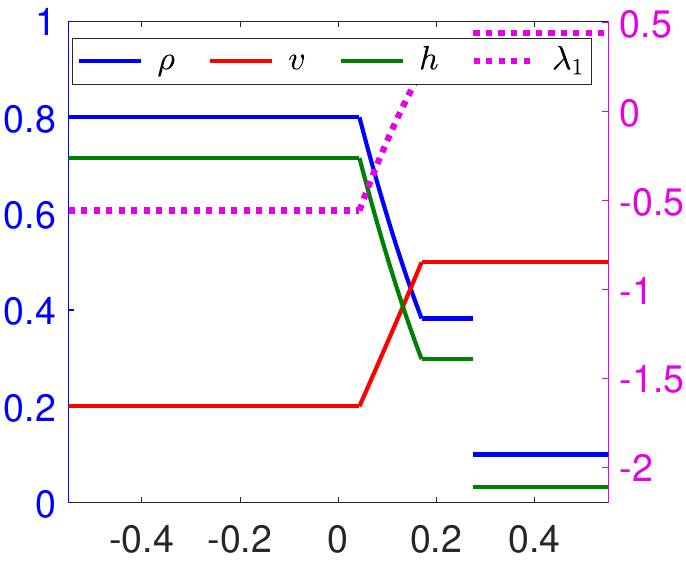}}	
		
	\end{center}
\end{minipage}

\vspace{4mm}

\begin{minipage}{0.49\textwidth}
	
	\begin{center}	
		\textbf{(iii)}\ \  
		$\boldsymbol{q(h)=h^{\nicefrac{1}{2}}}$, 
		$\boldsymbol{\h(\rho)=\rho^{2}}$
		\vspace{-1mm}			
		\scalebox{1}{\includegraphics[width=\linewidth]{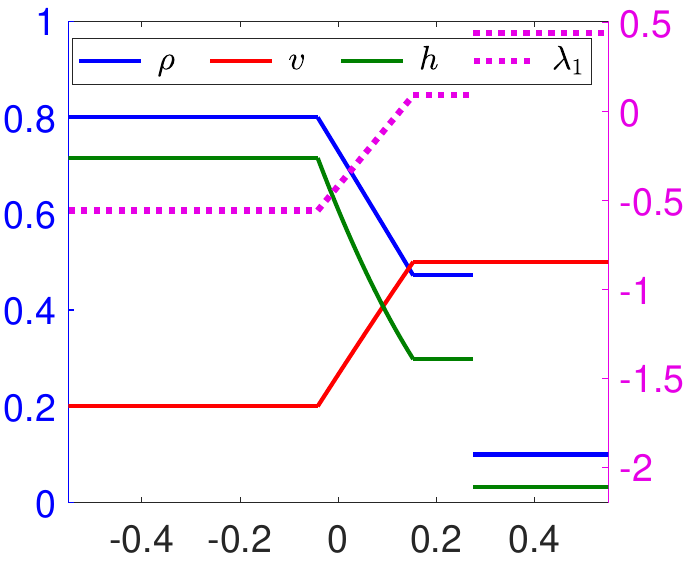}}

	\end{center}	
\end{minipage}
\hfil
\begin{minipage}{0.49\textwidth}
	
	\begin{center}
		\textbf{(iv)}\ \  
\textbf{comparison of velocities}		\vspace{-1mm}			
		\scalebox{1}{\includegraphics[width=\linewidth]{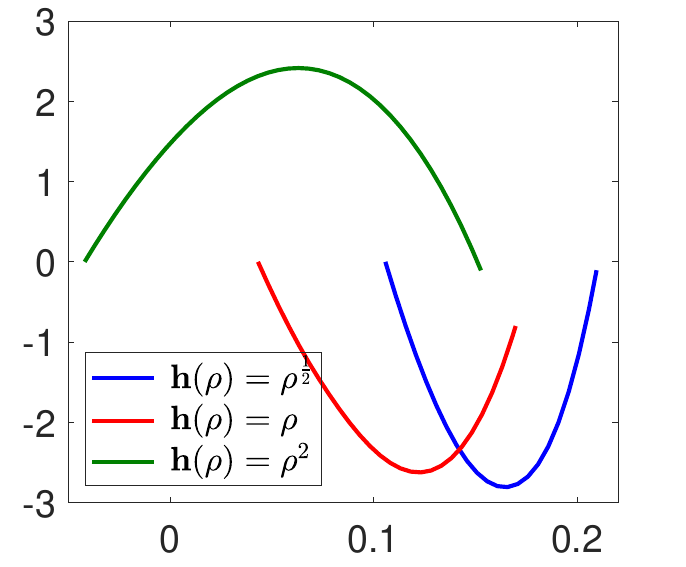}}	
		
	\end{center}
\end{minipage}


\caption{Panel (i)~--~(iii) state the solution to the Riemann problem at time $t=0.3$ for the choice~$q(h)=h^{\nicefrac{1}{2}}$. Panel (iv) highlights the difference~$\Delta v(t,x)$ over a rarefaction wave.}

\label{Figure1}	
\end{figure}

\begin{figure}[H]

	\begin{minipage}{0.49\textwidth}
		
		\begin{center}	
			\textbf{(i)}\ \  
			$\boldsymbol{q(h)=h}$, 
			$\boldsymbol{\h(\rho)=\rho^{\nicefrac{1}{2}}}$
			\vspace{-1mm}			
			\scalebox{1}{\includegraphics[width=\linewidth]{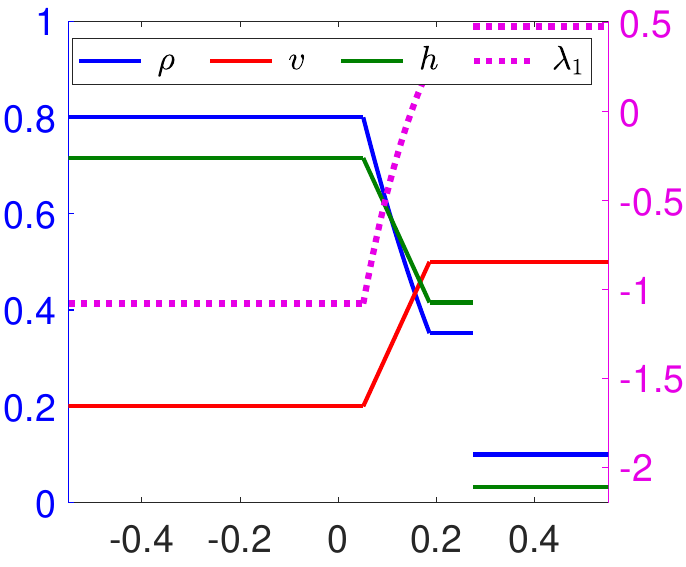}}

		\end{center}	
	\end{minipage}
	\hfil
	\begin{minipage}{0.49\textwidth}
		
		\vspace{-2mm}
		\begin{center}
			\textbf{(ii)}\ \  
			$\boldsymbol{q(h)=h}$, 
			$\boldsymbol{\h(\rho)=\rho}$
			\vspace{-1mm}			
			\scalebox{1}{\includegraphics[width=\linewidth]{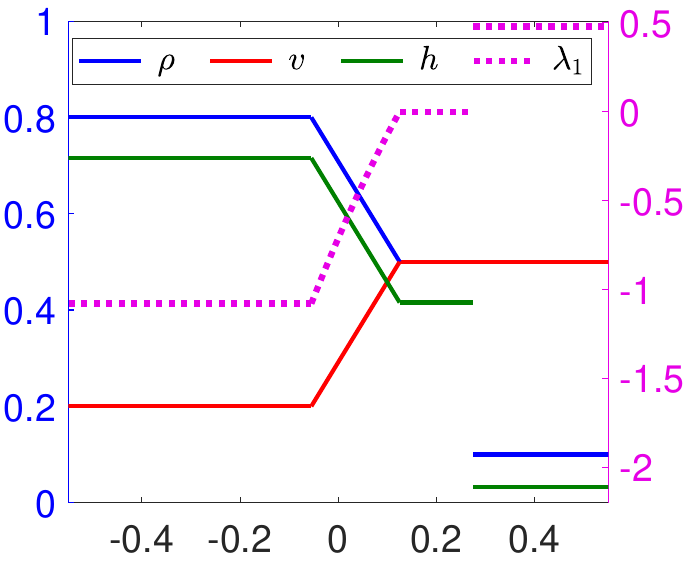}}	
			
		\end{center}
	\end{minipage}
	
	\vspace{4mm}
	
	\begin{minipage}{0.49\textwidth}
		
		\begin{center}	
			\textbf{(iii)}\ \  
			$\boldsymbol{q(h)=h}$, 
			$\boldsymbol{\h(\rho)=\rho^2}$
			\vspace{-1mm}			
			\scalebox{1}{\includegraphics[width=\linewidth]{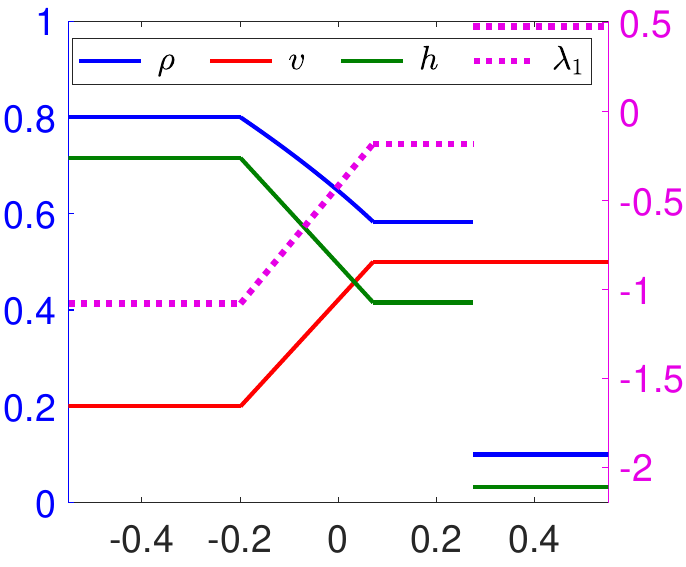}}

		\end{center}	
	\end{minipage}
	\hfil
	\begin{minipage}{0.49\textwidth}

		\begin{center}
			\textbf{(iv)}\ \  
\textbf{comparison of velocities}
        			\vspace{-1mm}			
			\scalebox{1}{\includegraphics[width=\linewidth]{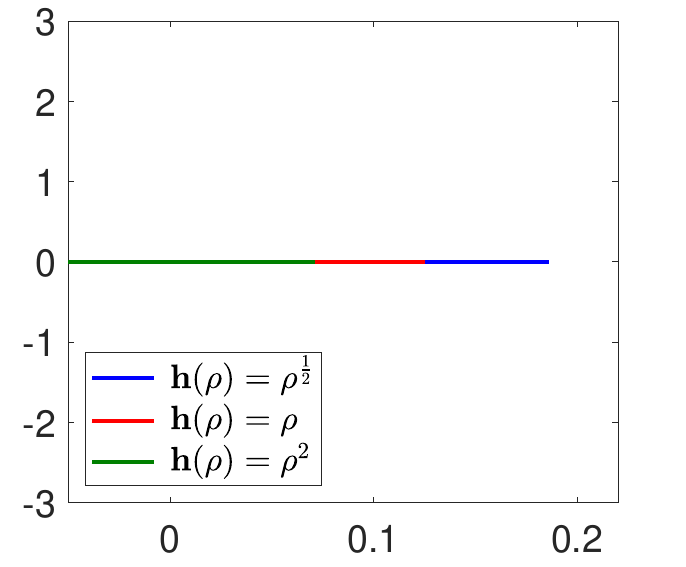}}	
			
		\end{center}
	\end{minipage}

\caption{Panel (i)~--~(iii) state the solution to the Riemann problem at time $t=0.3$ for the choice~$q(h)=h$, when the third-order model coincides with the Aw-Rascle-Zhang-model. Panel (iv) shows that the acceleration of drivers is constant over a rarefaction wave.}
	
\label{Figure2}		
\end{figure}

\begin{figure}[H]

	\begin{minipage}{0.49\textwidth}
		
		\begin{center}	
			\textbf{(i)}\ \  
			$\boldsymbol{q(h)=h^2}$, 
			$\boldsymbol{\h(\rho)=\rho^{\nicefrac{1}{2}}}$
			\vspace{-1mm}			
			\scalebox{1}{\includegraphics[width=\linewidth]{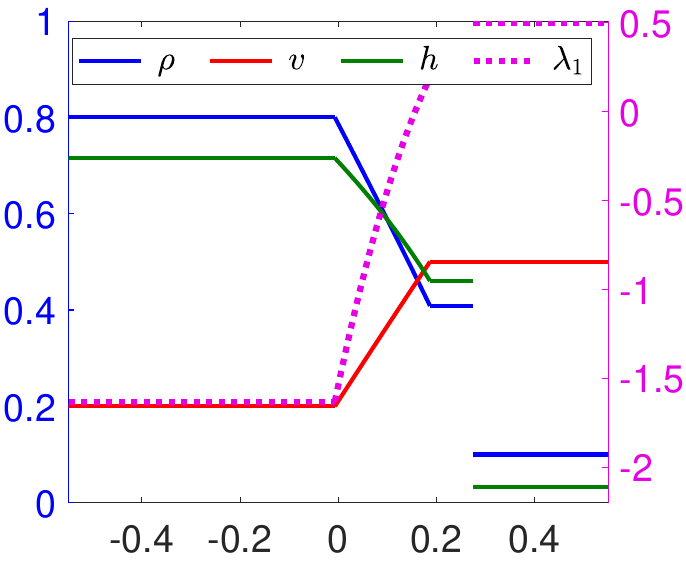}}

		\end{center}	
	\end{minipage}
	\hfil
	\begin{minipage}{0.49\textwidth}
		
		\vspace{-2mm}
		\begin{center}
			\textbf{(ii)}\ \  
			$\boldsymbol{q(h)=h^2}$, 
			$\boldsymbol{\h(\rho)=\rho}$
			\vspace{-1mm}			
			\scalebox{1}{\includegraphics[width=\linewidth]{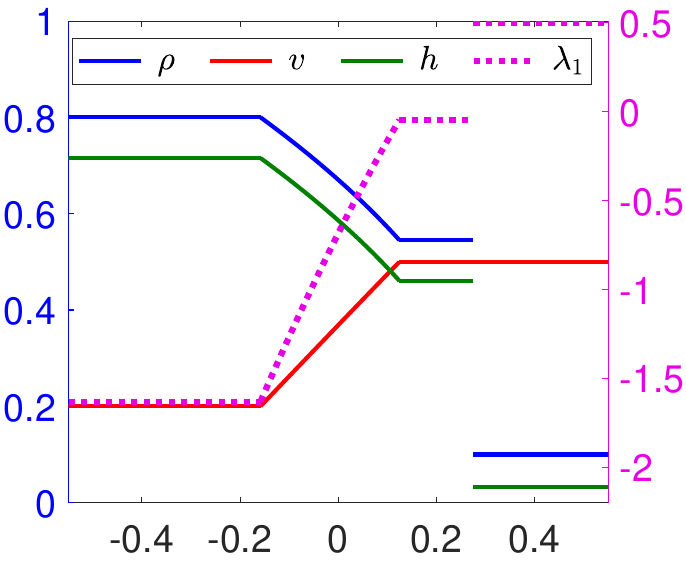}}	
			
		\end{center}
	\end{minipage}
	
	\vspace{4mm}
	
	\begin{minipage}{0.49\textwidth}
		
		\begin{center}	
			\textbf{(iii)}\ \  
			$\boldsymbol{q(h)=h^2}$, 
			$\boldsymbol{\h(\rho)=\rho^2}$
			\vspace{-1mm}			
			\scalebox{1}{\includegraphics[width=\linewidth]{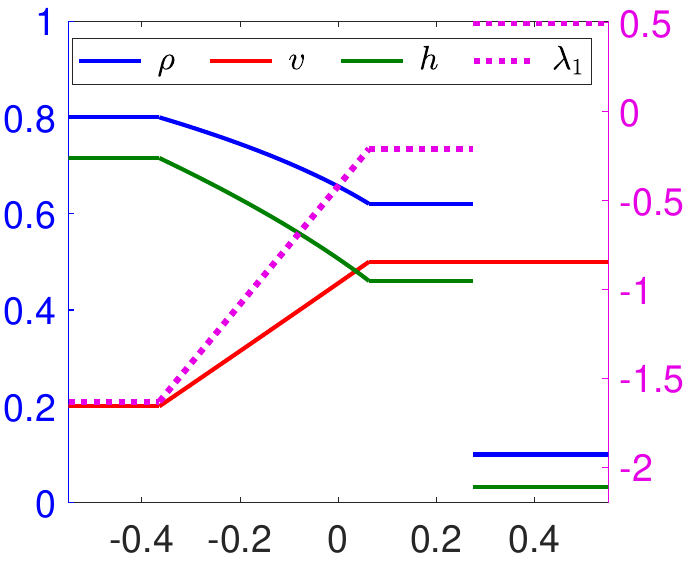}}

		\end{center}	
	\end{minipage}
	\hfil
	\begin{minipage}{0.49\textwidth}

		\begin{center}
			\textbf{(iv)}\ \  
\textbf{comparison of velocities}
			\vspace{-1mm}			
			\scalebox{1}{\includegraphics[width=\linewidth]{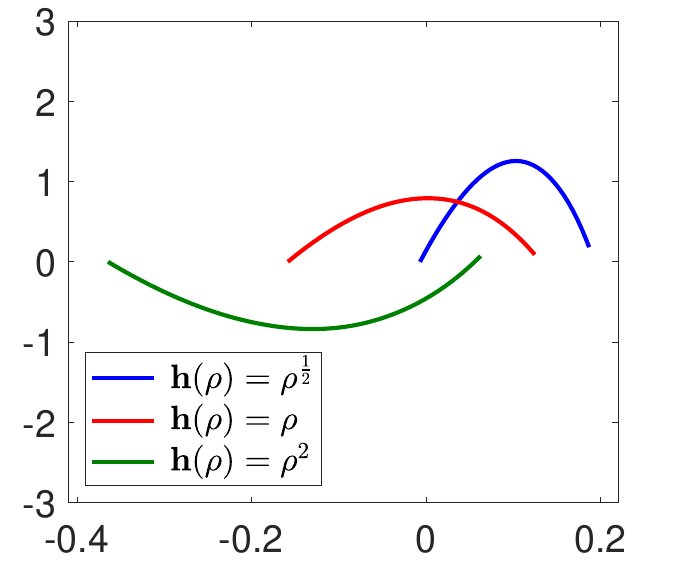}}	
			
		\end{center}
	\end{minipage}
	

\caption{Panel (i)~--~(iii) state the solution to the Riemann problem at time $t=0.3$ for the choice~$q(h)=h^2$. Panel (iv) highlights the difference~$\Delta v(t,x)$ over a rarefaction wave.}

\label{Figure3}		
\end{figure}

If the Rankine-Hugoniot condition~
$
f(\u_\ell) - f(\bar{\u})
=s(\u_\ell - \bar{\u})
$ holds, the left state is connected to the intermediate state by a shock wave with speed~$s$. Note that in this case the intermediate states~$\bar{v}$ and~$\bar{h}$ and can be obtained as in the previous discussion for the rarefaction wave, namely by making use of the Riemann invariants and the linear degeneracy of the third field.

Figure~\ref{FigureShock1} shows the solution to a Riemann problem with 
initial values 
$
\rho_\ell
=0.3
$, 
$
\rho_r
=0.7
$, 
$
v_\ell
=0.5
$, 
$
v_r
=0.2
$ and~$h_\ell=\rho_\ell^{\nicefrac{3}{2}}$, $h_r=\rho_r^{\nicefrac{3}{2}}$. 
Here, characteristics corresponding to the first field with eigenvalues~$\lambda_1(\u)$ cross and a shock occurs, which is followed by a contact corresponding to the second and third field.

\begin{figure}[H]

	\begin{minipage}{0.32\textwidth}
		
		\begin{center}	  
			$\boldsymbol{q(h)=h^{\nicefrac{1}{2}}}$, 
			$\boldsymbol{\h(\rho)=\rho^{\nicefrac{1}{2}}}$
			\vspace{-1mm}			
			\scalebox{1}{\includegraphics[width=\linewidth]{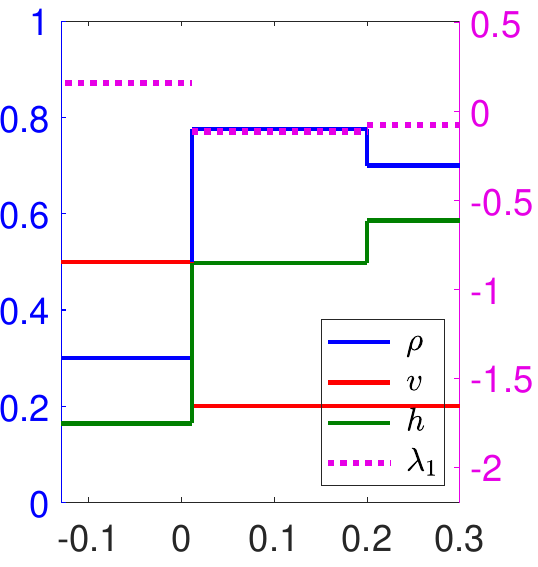}}

		\end{center}	
	\end{minipage}
	\hfil
	\begin{minipage}{0.32\textwidth}

		\begin{center}
			$\boldsymbol{q(h)=h^{\nicefrac{1}{2}}}$, 
			$\boldsymbol{\h(\rho)=\rho}$
			\vspace{-1mm}			
			\scalebox{1}{\includegraphics[width=\linewidth]{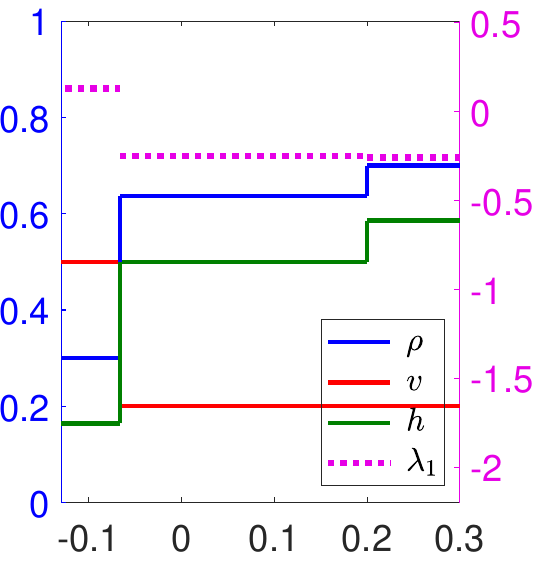}}	
			
		\end{center}
	\end{minipage}
	\hfil
	\begin{minipage}{0.32\textwidth}
		
		\begin{center}	
			$\boldsymbol{q(h)=h^{\nicefrac{1}{2}}}$, 
			$\boldsymbol{\h(\rho)=\rho^2}$
			\vspace{-1mm}			
			\scalebox{1}{\includegraphics[width=\linewidth]{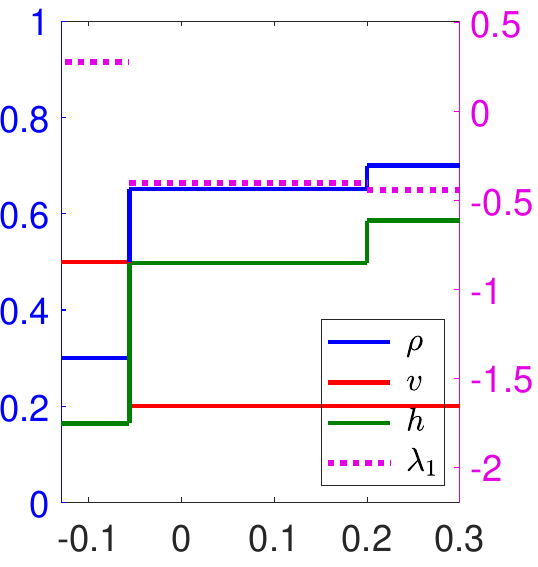}}

		\end{center}	
	\end{minipage}
	
	\vspace{4mm}
	
	\begin{minipage}{0.32\textwidth}
		
		\begin{center}	  
			$\boldsymbol{q(h)=h}$, 
			$\boldsymbol{\h(\rho)=\rho^{\nicefrac{1}{2}}}$
			\vspace{-1mm}			
			\scalebox{1}{\includegraphics[width=\linewidth]{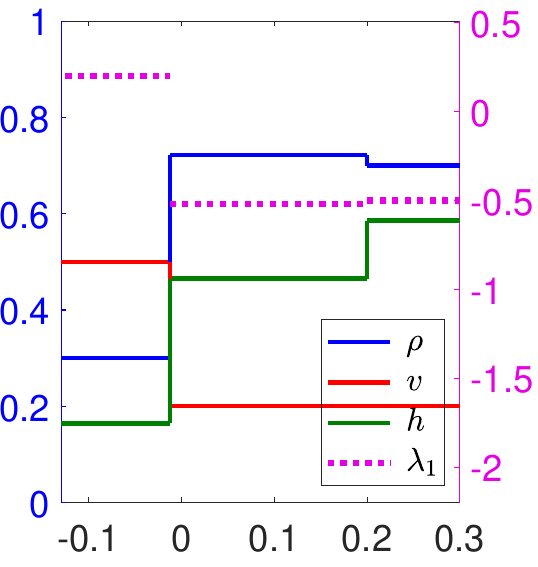}}

		\end{center}	
	\end{minipage}
	\hfil
	\begin{minipage}{0.32\textwidth}

		\begin{center}
			$\boldsymbol{q(h)=h}$, 
			$\boldsymbol{\h(\rho)=\rho}$
			\vspace{-1mm}			
			\scalebox{1}{\includegraphics[width=\linewidth]{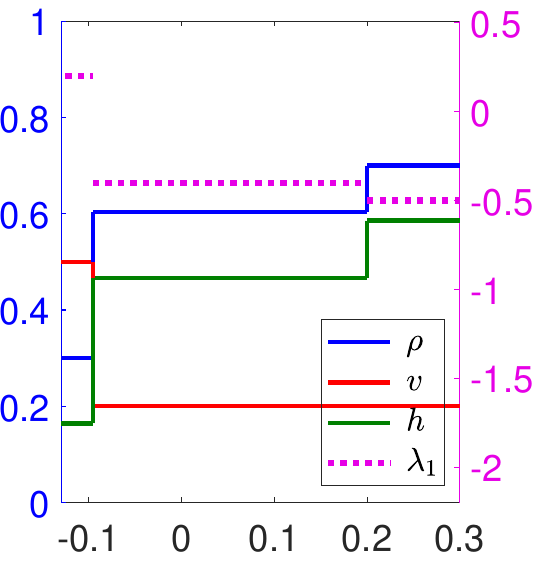}}	
			
		\end{center}
	\end{minipage}
	\hfil
	\begin{minipage}{0.32\textwidth}
		
		\begin{center}	
			$\boldsymbol{q(h)=h}$, 
			$\boldsymbol{\h(\rho)=\rho^2}$
			\vspace{-1mm}			
			\scalebox{1}{\includegraphics[width=\linewidth]{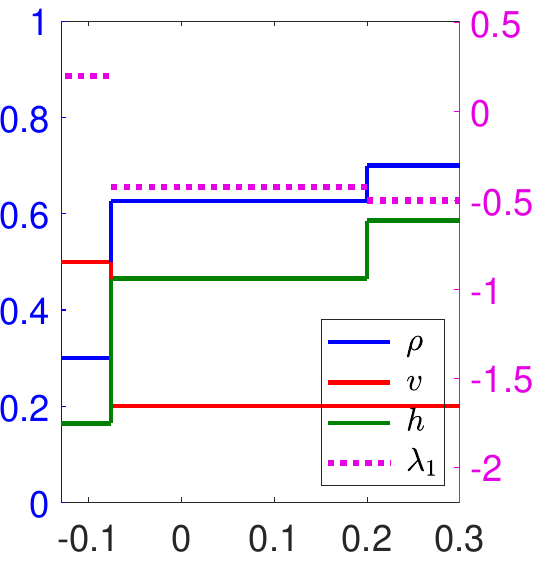}}

		\end{center}	
	\end{minipage}

\vspace{4mm}

	\begin{minipage}{0.32\textwidth}
		
		\begin{center}	  
			$\boldsymbol{q(h)=h^2}$, 
			$\boldsymbol{\h(\rho)=\rho^{\nicefrac{1}{2}}}$
			\vspace{-1mm}			
			\scalebox{1}{\includegraphics[width=\linewidth]{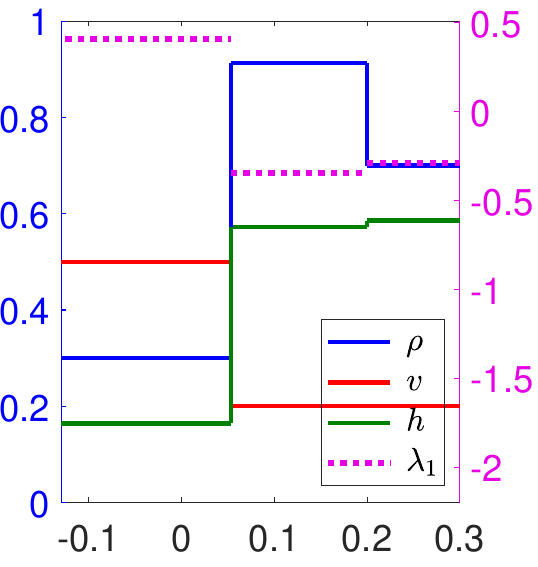}}

		\end{center}	
	\end{minipage}
	\hfil
	\begin{minipage}{0.32\textwidth}

		\begin{center}
			$\boldsymbol{q(h)=h^2}$, 
			$\boldsymbol{\h(\rho)=\rho}$
			\vspace{-1mm}			
			\scalebox{1}{\includegraphics[width=\linewidth]{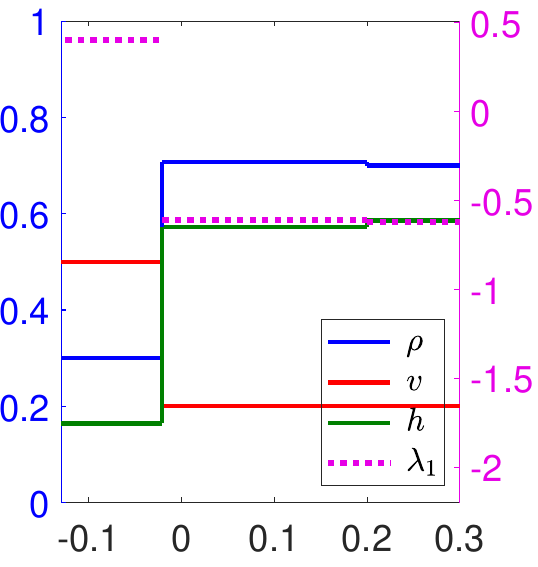}}	
			
		\end{center}
	\end{minipage}
	\hfil
	\begin{minipage}{0.32\textwidth}
		
		\begin{center}	
			$\boldsymbol{q(h)=h^2}$, 
			$\boldsymbol{\h(\rho)=\rho^2}$
			\vspace{-1mm}			
			\scalebox{1}{\includegraphics[width=\linewidth]{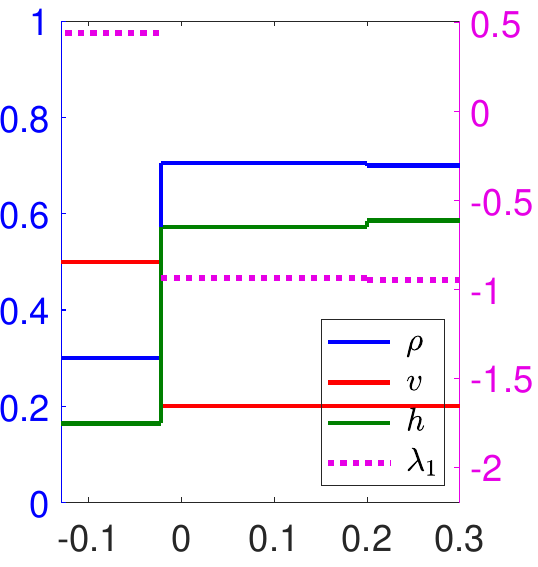}}

		\end{center}	
	\end{minipage}

\caption{Solutions to a Riemann problem with a shock, corresponding to the first field, and a contact with speed~$v_r$.}

\label{FigureShock1}	
\end{figure}

\section{Summary}
We have introduced a third-order traffic flow model. It has been justified by an amended and extended microscopic derivation of the Aw-Rascle-Zhang model. The hydrodynamic limit to the macroscopic level  results in a hyperbolic Temple class system. A relaxation to the Aw-Rascle-Zhang model has been introduced and a stability analysis has been performed. 
Chapman-Enskog-type expansions have been used as basic tool, which also yield conditions to properly choose model and relaxation parameters. 
Furthermore, relations to multiclass systems, which are based on velocity curves that are parameterized by  driver-dependent quantities,   have been established and the Riemann problem is discussed.

\section*{Appendix}

The following lemma states the equivalence between the formulations in Eulerian and Lagrangian coordinates. The proof follows basic arguments taken from~\cite{Aw:2002,Bagnerini:2003}, which are based on smoothness assumptions. However, we note that equivalence of hyperbolic systems in  Eulerian and Lagrangian coordinates can also be obtained in terms of weak entropy solutions~\cite{Wagner1987}. 
\medskip

\begin{lemma}
	The systems~\eqref{PiuLagrange} in Lagrangian and~\eqref{PiuS1} in Eulerian coordinates are equivalent. 
\end{lemma}

\medskip
\begin{proof}
	The transforms~\eqref{TransformsLagrange}     
	yield the expressions
	\begin{alignat*}{8}
		&	\partial_t \rho(t,x)
		&&=
		\partial_T \rho(T,X)
		+
		\partial_X \rho(T,X)
		\partial_t X(t,x)
		&&=
		\partial_T \rho(T,X)
		-
		\rho v
		\partial_X \rho(T,X),
		\\
		&	\partial_x \rho(t,x)
		&&=
		\partial_X \rho(T,X)
		\partial_x X(t,x)
		&&=
		\rho\partial_X \rho(T,X), \\
		&	\partial_T \rho(T,X)
		&&	=	
		\partial_t \rho(t,x)
		+
		\partial_x \rho(t,x)
		\partial_T x(T,X)
		&&	=
		\partial_t \rho(t,x)
		+
		v
		\partial_x \rho(t,x),\\
		&\partial_X \rho(T,X)
		&&=
		\partial_x \rho(t,x)
		\partial_X x(T,X)
		&&=
		\frac{1}{\rho}
		\partial_x \rho(t,x).
	\end{alignat*}

	\noindent
	The continuity equation reads as
	\begin{align*}
&	0=
	\partial_t \rho
	+\partial_x (\rho v)
	=
	\partial_T \rho
	-
	\rho v
	\partial_X \rho
	+
	\rho
	\partial_X
	(\rho v)
	=
	\partial_T\rho+
	\rho^2\partial_Xv \\
	&\Leftrightarrow\quad
	\partial_T \tau - \partial_X v=0
	\quad\text{for}\quad
	\tau=\frac{1}{\rho}.
	\end{align*}
	The third equation, namely the material derivative, satisfies~$
	\partial_T\big(
	h-\h(\rho)
	\big)=0
	$. 
	The ansatz~$\omega= v+q(h)$ yields
	\begin{align*}
		&\partial_T q(h)
		=
		q'(h)
		\partial_T h 
		=
		q'(h)
		\big(
		\partial_t \h(\rho)
		+v
		\partial_x \h(\rho)
		\big)
		=
		-
		q'(h)
		\h'(\rho)
		\rho \partial_x v,\\
		&
		0
		=
		\partial_T \Big( v+q(h) \Big)
		=
		\partial_t 
		v
		+
		\Big(
		v
		-
		q'(h)
		\h'(\rho)\rho
		\Big)
		\partial_x v.
	\end{align*}
\hfill	
\end{proof}

\bigskip

\section*{Acknowledgments}
This work was carried out within the activities of the PRIN PNRR Project 2022 No.~P2022JC95T ``Data-driven discovery and control of multiscale interacting artificial agent systems'', funded by MUR (Ministry of University and Research) and Next Generation EU -- European Commission.\\
The authors acknowledge the support of Sapienza University under Ateneo Project 2023 ``Modeling, numerical treatment of hyperbolic equations and optimal control problems''.\\
GV is member of the INdAM Research National Group of Scientific Computing (INdAM-GNCS).

\bigskip

\bibliographystyle{plain}
\bibliography{trafficbib}

          \end{document}